\documentclass[11pt,a4paper,reqno]{amsart}

\usepackage{amssymb,amsmath,amsthm}

\usepackage[mathscr]{eucal} 
\usepackage{graphicx}
\usepackage[hidelinks]{hyperref}
\usepackage{xcolor}
\hypersetup{
 colorlinks,
 linkcolor={red!50!black},
 citecolor={blue!50!black},
 urlcolor={blue!80!black}
}
\usepackage{pdfsync}
\usepackage{enumerate} 
\usepackage{esint}

\theoremstyle{plain}
\newtheorem{theorem}{Theorem}[section]

\newtheorem{lemma}[theorem]{Lemma}
\newtheorem{corollary}[theorem]{Corollary}
\newtheorem{proposition}[theorem]{Proposition}
\theoremstyle{remark}
\newtheorem{remark}[theorem]{Remark}

\numberwithin{equation}{section}
\newcommand{\Z}{\mathbb{Z}}

\usepackage{dsfont}

\allowdisplaybreaks

\def\C{{\mathbb C}}
\def\R{{\mathbb R}}
\def\T{{\mathbb T}}
\def\N{{\mathbb N}}
\def\Z{{\mathbb Z}}

\newcommand{\dd}{\,\mathrm{d}}
\newcommand{\loc}{\mathrm{loc}}
\newcommand{\surf}{\mathrm{surf}}
\newcommand{\app}{\mathrm{app}}
\newcommand{\vecspan}{\mathrm{Span}}

\renewcommand{\div}{\operatorname{div}}

\renewcommand{\leq}{\leqslant}\renewcommand{\le}{\leqslant}
\renewcommand{\geq}{\geqslant}
\renewcommand{\epsilon}{\varepsilon}
\newcommand{\norm}[1]{\left\lVert#1\right\rVert}

\def\Tend#1#2{\mathop{\longrightarrow}\limits_{#1\rightarrow#2}}

\def\cal#1{\mathcal{#1}}

\title[Spectral gaps for linearized water waves]{Bloch-Floquet band gaps for water waves over a periodic bottom}

\author[C. Lacave]{Christophe Lacave}
\author[M. M\'enard]{Matthieu M\'enard}
\author[C. Sulem]{Catherine Sulem}

\address[C. Lacave]{CNRS, LAMA, ISTerre, Univ. Savoie Mont Blanc, 73000 Chamb\'ery, France.}
\email{Christophe.Lacave@univ-smb.fr}

\address[M. M\'enard]{D\'epartement de math\'ematiques, Universit\'e Libre de Bruxelles, 1050 Brussels, Belgium.}
\email{matthieu.menard@ulb.be}

\address[C. Sulem]{Department of Mathematics, University of Toronto, M5S 2E4 Toronto Ontario, Canada.}
\email{sulem@math.utoronto.ca}

\date{\today}

\begin{document}
\maketitle

\begin{abstract}
A central object in the analysis of the water wave problem is the Dirichlet-Neumann operator. This paper is devoted to the study of its spectrum in the context of the water wave system linearized near equilibrium in a domain with a variable bottom, assumed to be a $C^2$ periodic function. We use the analyticity of the Dirichlet-Neumann operator with respect to the bottom variation and combine it with general properties of elliptic systems and spectral theory for self-adjoint operators to develop a Bloch-Floquet theory and describe the structure of its spectrum. We find that under some conditions on the bottom variations, the spectrum is composed of bands separated by gaps, with explicit formulas for their sizes and locations. 
\end{abstract}

\textit{\small We dedicate this article to the memory of Thomas Kappeler. This work started as a collaborative project with Thomas who sadly left us abruptly. We will always be grateful for his inspiration, generosity and kindness. }

\section{Introduction}

This study concerns the motion of a free surface wave over a variable bottom. There is a large literature devoted to this subject due to its relevance to oceanography in coastal engineering. Formation of long-shore sandbars along gentle beaches has been observed in open ocean coasts or bays and it is important to understand how they affect the propagation of waves \cite{MSY05-P1}. For mathematical purposes, the variable bottom is often assumed to be periodic or described by a stationary random process. The effect of a fast oscillating bottom has been studied in many asymptotic regimes \cite{RP83, CGNS05, GKN07,CLS12} where effective equations are derived using techniques of homogenization and of multiple scales. Here, we restrict ourselves to the linearized water wave problem near the equilibrium with a periodic bottom. In this setting, there is a classical phenomenon known as the Bragg resonance reflection phenomenon, in analogy with the Bragg's law for X-rays in crystallography. In the water wave setting, it refers to the situation where the bottom has the form $y =h(x)$, where $h(x) =h_1$ (constant) if $x\le 0$ and $x\geq \ell$ and is a periodic function of period $d$ for $0<x<\ell$. Resonance happens when the wavelength $\lambda$ of the incident wave is equal to twice that of the bottom variation with higher-order resonances for $\lambda = 2d/n, n = 2, 3$... This configuration leads to strong reflected waves. This phenomenon was observed experimentally by Heathershaw \cite{H82} and derived formally by Mei \cite{M85} and Miles \cite{M98}. Based on this analysis, Mei \cite{M85} proposed a theory that strong reflection can be induced by sand bars if the Bragg resonance conditions are met, thus protecting the beach from the full impact of the waves. Higher-order Bragg reflections have been observed experimentally and numerically by Guazzelli, Rey and Belzons \cite{GRB92}.

The problem under consideration is wave propagation over a periodic bottom of infinite extension. Porter and Porter \cite{PP03} draw a comparison between scattering by a finite length periodic bottom and infinite length periodic bottom.
The two problems have however a different character, the finite extension one giving rise to a boundary-value problem and the phenomenon of Bragg resonances while the infinite extension one takes the form of an eigenvalue problem. 
They first observed that the scattering properties for a bottom with an arbitrary number of periods may be obtained from those for the single period through a transfer matrix technique. On the other hand, they defined an extended Bloch problem where they not only look for eigenvalues corresponding to propagating waves over the infinite periodic bottom but also eigenvalues corresponding to evanescent waves, which increase or decrease as they cross a single period.
They show that the eigenvalues of the extended transfer matrix used in the scattering case approximate those of the extended Bloch problem (see also Yu and Howard \cite{YH12}). Other authors pointed out the close relation between the problems of wave propagation over a periodic bottom of finite or infinite extension respectively. We mention the work of Linton \cite{L11} who observed a band-gap structure 
in the context of water waves propagating over infinite periodic arrays of submerged horizontal circular cylinders in deep water. He showed that the approximate location of the band gaps can be obtained from the phase of the transmission coefficient for a single cylinder. In a more recent reference, Liu, Liu and Lin \cite{LLL19}, considered shallow water waves over infinite arrays of parabolic bars. They explicitly calculated the band gaps and analyzed the influence of the bar height, width and spacing on band gaps. They found that, if a band gap exists for the spectral problem with an infinite periodic array, it gives rise for the associated problem with a finite periodic extension, to a Bragg resonant reflection occuring at the mid-point of the first band gap, providing a more accurate prediction that the classical Bragg's law.

\subsection{Setting of the problem}

The goal of this paper is a detailed study of the Dirichlet-Neumann operator and in particular the description of its spectrum for general periodic topography through an analysis of associated elliptic systems and Bloch-Floquet theory.

The starting point of our analysis is the water wave problem written in its Hamiltonian formulation \cite{Z68, CS93}.
The two dimensional fluid domain is
\[
 {\Omega}_\epsilon(b,\eta) = \left\{(x, z): x \in \R, -h + \epsilon b(x) < z< \eta(x,t) \right\}
\]
where the variable bottom is given by $z =-h +\epsilon b(x)$, and the free surface elevation by $z= \eta(x,t)$. The water wave problem in canonical variables $(\eta, \xi)$, where $\xi(x) $ is the trace of the velocity potential on the free surface $\{z=\eta(x,t)\}$ has the form 
\begin{align}\label{WWeq}
	\left\{ 
	\begin{array}{l}
	\partial_t \eta -G[\eta, \epsilon b] \xi =0 , \\
	\partial_t \xi + g\eta +\frac{1}{2} |\partial_x\xi|^2 -
 \dfrac{\left(G[\eta, \epsilon b] \xi + \partial_x \eta\cdot \partial_x
 \xi\right)^2}{2(1+ |\partial_x\eta|^2)} = 0 ,
\end{array}
\right .
\end{align}
where $g$ is the acceleration due to gravity. 
The operator $G[\eta,\epsilon b] $ is the Dirichlet-Neumann operator, defined by
\begin{equation*}
	G[\eta,\epsilon b]\xi =
	\sqrt{1+|\partial_x\eta|^2}\partial_n\Phi_{\vert_{z = \eta}} ,
\end{equation*}
where $\Phi$ is the solution of the elliptic boundary value problem
\begin{equation*}
	\left\lbrace 
	\begin{array}{l}
	\partial_x^2\Phi+\partial_z^2\Phi = 0 \quad 
 \text{in} \quad {\Omega}_\epsilon(b,\eta) , \\
	\Phi_{\vert_{z = \eta}} = \xi,\qquad
 \partial_n\Phi_{\vert_{z = -h+\epsilon b }} = 0 ,
	\end{array}\right. 
\end{equation*}
The system \eqref{WWeq}, linearized about the stationary solution $(\eta(x),\xi(x))=(0,0)$ is 
\begin{equation*}
	\begin{cases}
	\partial_t \eta -G[\epsilon b] \xi =0 , \\
	\partial_t \xi + g\eta =0 .
\end{cases}
\end{equation*}
where now, and for the remainder of this article, we denote $G[0,\epsilon b]$ by $G[\epsilon b]$. 
The surface elevation $\eta$ satisfies
\begin{equation*}
 \partial^2_{t} \eta + g G[\epsilon b] \eta = 0 ,
\end{equation*}
and initial conditions 
\begin{equation*}
\eta(x,0) = \eta_0(x), \;\; \partial_t\eta(x,0) = \eta_1(x), \;\; x \in \R .
\end{equation*}
When we look for solutions of the form $ \eta(x,t) = e^{i\omega t} v(x)$, we are led to the spectral problem $G[\epsilon b] v = \lambda v$ with $\lambda = \omega^2/g$. The operator $G[\epsilon b]$ is a nonlocal operator depending on the function $b(x)$, that we assume to be a $2\pi$-periodic $C^2$ function.
In analogy with second-order differential operator with periodic coefficients, the goal is to develop a Bloch-Floquet decomposition to describe its spectrum.

We will recall in Section~\ref{sec:Bloch} the principle of the Bloch-Floquet theory, where one decomposes a function as an integral of $\theta$-periodic functions, namely one writes $f(x)=\int_{-1/2}^{1/2}\cal{U} f(x,\theta)\dd \theta$ where $\cal{U} f$ is $\theta$-periodic \begin{equation*}
 \cal{U} f(x+2\pi,\theta) = e^{2\pi i\theta}\cal{U} f(x,\theta), 
\end{equation*}
 for the Bloch-Floquet parameter $\theta \in (-\frac{1}{2},\frac{1}{2}]$.
The principle of the Bloch-Floquet decomposition is to describe the spectrum and the generalized eigenfunctions of $G[\varepsilon b]$ by a family, parametrized by $\theta$, of spectral problems for $G_\theta[\varepsilon b]$ acting on $2\pi$ periodic functions:
\begin{equation}\label{ev-pb-per}
 \begin{cases}
 G_\theta[\epsilon b] \phi^{\epsilon}(x,\theta) := e^{-i\theta x} G[\epsilon b] e^{i\theta x} \phi^{\epsilon}(x,\theta) 
 = \lambda^{\epsilon}(\theta) \phi^{\epsilon}(x,\theta) , \\
 \phi^{\epsilon}(x+2 \pi, \theta) = \phi^{\epsilon}(x,\theta) .
 \end{cases}
\end{equation}
Equivalently, $G_\theta[\epsilon b]$ is defined in terms of the elliptic system \eqref{D2N-theta} by \eqref{DN0-theta-def}.

The goal of this paper is to describe rigorously the spectrum of the Dirichlet-Neumann operator in the form of bands separated by gaps, thus expressing the wave elevation as a sum of Bloch waves. 
This decomposition shows that there are band gaps within which Bloch waves cannot exist.
A classical model operator that exhibits this behavior is the Hill operator $H=-\frac{\dd}{\dd x^2} +V(x)$, where $V(x)$ is a smooth $2\pi$-periodic potential on $\R$. The associated spectral problem is 
\[
-\frac{\dd\varphi}{\dd x^2} +V(x)\varphi =\lambda \varphi ,
\]
which has been intensively studied. We refer to the books of Eastham \cite{E73}, Reed and Simon \cite{RS78} and the detailed review of Kuchment \cite{K16} and references therein.

When the bottom is flat, $b=0$, the eigenvalues $\kappa_p(\theta)$ of $G_\theta[0]$ are given explicitly in terms of the dispersion relation for water waves over a constant depth $h=1$ and $\varepsilon=0$:
\begin{equation*} 
 \kappa_p(\theta)=\frac{\omega^2(p+\theta)}g = (p+\theta)\tanh(p+\theta),
\end{equation*}
for $p \in \Z$, and Bloch-Floquet parameter $\theta \in (-\frac{1}{2},\frac{1}{2}]$. Eigenvalues are simple for $-1/2 < \theta < 0$ and $0 < \theta < 1/2$. For $ \theta = 0, 1/2$, they have multiplicity two. When reordered appropriately by their size, the eigenvalues, denoted $ \lambda_p^{0}(\theta)$, are continuous in $\theta$ (see Figure~\ref{fig}.a). The spectrum of the Dirichlet-Neumann operator $G[0]$ is the half-line $[0, +\infty)$. The goal of this work is to understand how the presence of a small periodic bottom modifies the structure of the Dirichlet-Neumann operator.

\subsection{Main results}
We will prove that, under certain conditions on the Fourier coefficients of $b$, the presence of the bottom generally results in the splitting of double eigenvalues near points of multiplicity, creating a spectral gap.
Yu and Howard \cite{YH12} computed numerically Bloch eigenfunctions and eigenvalues of \eqref{ev-pb-per} for various examples of bottom profiles using a conformal map that transforms the original fluid domain to a uniform strip, thus identifying the corresponding spectral gaps. 
Chiad\`o Piat, Nazarov and Ruotsalainen \cite{CNR13} gave a necessary and sufficient condition on the Fourier coefficients of the bottom variations to ensure the opening of a finite number of spectral gaps of $\mathcal{O}(\epsilon)$. In \cite{CGLS18}, a systematic method, based on the Taylor expansion of the Dirichlet-Neumann operator in powers of $b$, was proposed to compute explicitly spectral gaps, allowing spectral gaps of high order. A simple example of bottom topography was given leading to gaps of order $\mathcal{O}(\epsilon^4)$.
In this paper, we give a full description of the spectrum of $G[\epsilon b]$.
We first prove that it is purely absolutely continuous and is composed of union of bands.
 We then give necessary and sufficient conditions on the Fourier coefficients of $b(x)$ for the opening of gaps of order $\epsilon$ and $\epsilon^2$, based on a rigorous perturbation theory near double eigenvalues of the unperturbed problem.

The main ingredients of our analysis are elliptic estimates \cite{Lannes}, perturbation theory of self adjoint operators \cite{Rellich, RS78,L22} and the notion of quasi-modes which provides, under some conditions, a method to construct eigenvalues from approximate ones \cite[Proposition~5.1]{BKP15}. It also strongly relies on the
analyticity of the $G_\theta[\epsilon b]$ and its resolvent with respect to $\epsilon$ and $\theta$.

\begin{theorem}[Structure of the spectrum] \label{union-of-bands} 
Let $b\in C^2(\T_{2\pi})$. There exists $\varepsilon_{0}(b)>0$ such that the following holds true for any $\varepsilon\in [0,\varepsilon_{0})$.
\begin{enumerate}[(i)]
 \item The spectrum $\sigma(G[\epsilon b])$ is composed of a union of bands. Namely,
\begin{equation*}
 \sigma(G[\epsilon b]) = \bigcup_{p=0}^{\infty}\lambda_p^{\epsilon}\Big((-\tfrac{1}{2},\tfrac{1}{2}]\Big) ,
\end{equation*}
where the $\{\lambda_p^{\epsilon}(\theta) \}_{p=0}^\infty$ are the eigenvalues of $G_\theta[\epsilon b]$, labeled in increasing order, repeated with their order of multiplicity, and the bands are images of the Lipschitz non-negative functions $\theta\mapsto \lambda_p^{\epsilon}(\theta)$ on the interval $(-\frac12,\frac{1}{2}]$. Moreover, $\lambda_0^\epsilon(0) = 0$.

\item For any $p\in \N$, there exists $\varepsilon_{1}(b,p)\in (0,\epsilon_0]$ and $C_{b,p}$ such that we have
\begin{equation*}
 \dd \Big(\lambda^{\epsilon}_p\Big( ( -\tfrac{1}{2},\tfrac{1}{2} ] \Big),\lambda^{0}_p\Big(( -\tfrac{1}{2},\tfrac{1}{2} ]\Big)\Big) \leq C_{b,p} \epsilon, \quad \forall \varepsilon\in [0,\varepsilon_{1}).
\end{equation*}

\item The lower part of the spectrum of $G[\epsilon b]$ is purely absolutely continuous\footnote{For a precise definition of the absolutely continuous spectrum we refer to \cite[Chapter VII.2]{RS80}.}. More precisely, for any $M>0$, there exists $\varepsilon_{M}\in (0,\epsilon_0]$ such that for any $\varepsilon\in [0,\varepsilon_{M})$,
\begin{align*}
\sigma(G[\epsilon b])\cap [0,M]&=\sigma_{\rm ac}(G[\epsilon b])\cap [0,M] \\
\sigma_{\rm pp}(G[\epsilon b])\cap [0,M] &= \sigma_{\rm sc}(G[\epsilon b])\cap [0,M] = \emptyset.
\end{align*}
\end{enumerate}
\end{theorem}

The next results give conditions on the Fourier coefficients of $b$, defined as
\[
\widehat{b}_p= \frac1{2\pi}\int_0^{2\pi} b(x) e^{-ipx} \dd x,
\]
that ensure the opening of a gap that separates the double eigenvalues $\lambda^{0}_{2p-1}(0)=\lambda_{2p}^{0}(0)$ or $ \lambda^{0}_{2p}(\frac{1}{2})=\lambda_{2p+1}^{0}(\frac{1}{2}) $
corresponding to $b=0$. Let us denote
\begin{equation}\label{def-Fp}
F_p := \frac{\big(\frac{p}{2}\big)^2}{\cosh^2\big(\frac{p}{2}\big)} 
\end{equation}

\begin{theorem}[Gap opening of order $\epsilon$] \label{one_gap-order-eps}
Let $b\in C^2(\T_{2\pi})$ and $p\in \N$. There exist positive numbers $\epsilon_2(b,p)$ and $C_{b,p}$ such that the following holds true.
\begin{enumerate}[(i)]
 \item If $p>0$ and $\widehat{b}_{2p} \neq 0$, then, for all $\epsilon \in (0,\epsilon_2)$, the spectrum $\sigma(G[\epsilon b])$ has a gap: 
\[
\lambda_{2p}^{0}(0) - g^{-}_{2p,\epsilon} := \max_{-\frac12 < \theta\le \frac12}\lambda_{2p-1}^{\epsilon}(\theta)
 < \min_{-\frac12 < \theta\le \frac12} \lambda_{2p}^{\epsilon}(\theta) 
 =:
 \lambda_{2p}^{0}(0) + g^{+}_{2p,\epsilon}
\]
with 
\begin{equation*}
 \Big| g^{\pm}_{2p,\epsilon} - F_{2p} |\widehat{b}_{2p}|\epsilon \Big| 
 \le C_{b,p} \epsilon^{2}. 
\end{equation*}
\item If $\widehat{b}_{2p+1} \neq 0$, then, for all $\epsilon \in (0,\epsilon_2)$, the spectrum $\sigma(G[\epsilon b])$ has a gap:
\begin{equation*}
 \lambda_{2p}^{0}(\tfrac12) - g^{-}_{2p+1,\epsilon} := \max_{-\frac12 < \theta\le \frac12}\lambda_{2p}^{\epsilon}(\theta)
 < \min_{-\frac12 < \theta\le \frac12} \lambda_{2p+1}^{\epsilon}(\theta) 
 =:
 \lambda_{2p}^{0}(\tfrac12) + g^{+}_{2p+1,\epsilon}
\end{equation*}
with 
\begin{equation*}
 \Big| g^{\pm}_{2p+1,\epsilon} - F_{2p+1} |\widehat{b}_{2p+1}|\epsilon \Big| 
 \le C_{b,p} \epsilon^{2}.
\end{equation*}
\end{enumerate}
\end{theorem}

If $\widehat{b}_{2p} =0$ and if another condition on the Fourier coefficients of $b$ is satisfied, then a gap of size $\epsilon^2$ occurs.
Let us denote 
\begin{align*}
J_p(b) &= \frac{p^2}{\cosh(p)^2}\sum_{k \notin \{0,p,-p\}}\frac{k^2-\kappa_k(0) \kappa_p(0) }{\kappa_p(0) - \kappa_k(0)}|\widehat{b}_{k-p}|^2 \\
S_p(b) &= \frac{p^2}{\cosh(p)^2}\sum_{k \notin \{0,p,-p\}}\frac{k^2-\kappa_k(0)\kappa_p(0)}{\kappa_p(0) - \kappa_k(0)}\widehat{b}_{k+p}\overline{\widehat{b}_{k-p}}.
\end{align*}

\begin{theorem}[Gap opening of $\epsilon^2$] \label{one_gap-order-eps2}
Let $b\in C^2(\T_{2\pi})$ and $p\in \N^*$. There exist positive numbers $\epsilon_3(b,p)$ and $C_{b,p}$ such that the following holds true.
If $\widehat{b}_{2p} = 0$ and $S_p(b) \ne 0$, then, for all $\epsilon \in (0,\epsilon_3)$, the spectrum $\sigma(G[\epsilon b])$ has a gap:
 \begin{multline*} 
 \lambda_{2p}^{0}(0) + J_p 
 \epsilon^2- g^{-}_{2p,\epsilon} := \max_{-\frac12 < \theta\le \frac12}\lambda_{2p-1}^{\epsilon}(\theta)\\
 < \min_{-\frac12 < \theta\le \frac12} \lambda_{2p}^{\epsilon}(\theta) 
 =:
 \lambda_{2p}^{0}(0) + J_p \epsilon^2 + g^{+}_{2p,\epsilon}
\end{multline*}
with 
\begin{equation*} 
\Big| g ^{\pm}_{2p,\epsilon} - |S_p| \epsilon^2 
\Big| \le C_{b,p} \epsilon^{3} .
\end{equation*}
\end{theorem}

If $\widehat{b}_{2p+1} = 0$, similar conditions on the Fourier coefficients of $b$ lead to the opening of a gap of order $\epsilon^2$ near $\theta =\pm \frac{1}{2}$.

\begin{figure}[h!t]
\centering{(a){\includegraphics[height=0.35\textwidth]{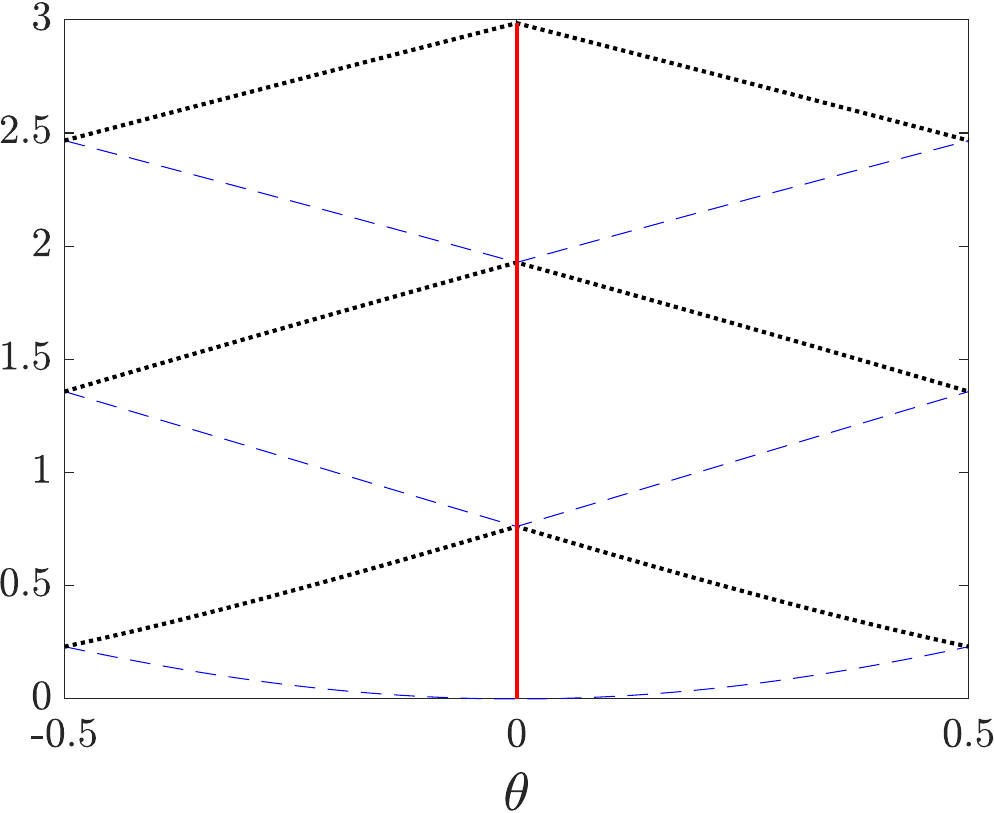}}
 \ \ (b){\includegraphics[height=0.35\textwidth]{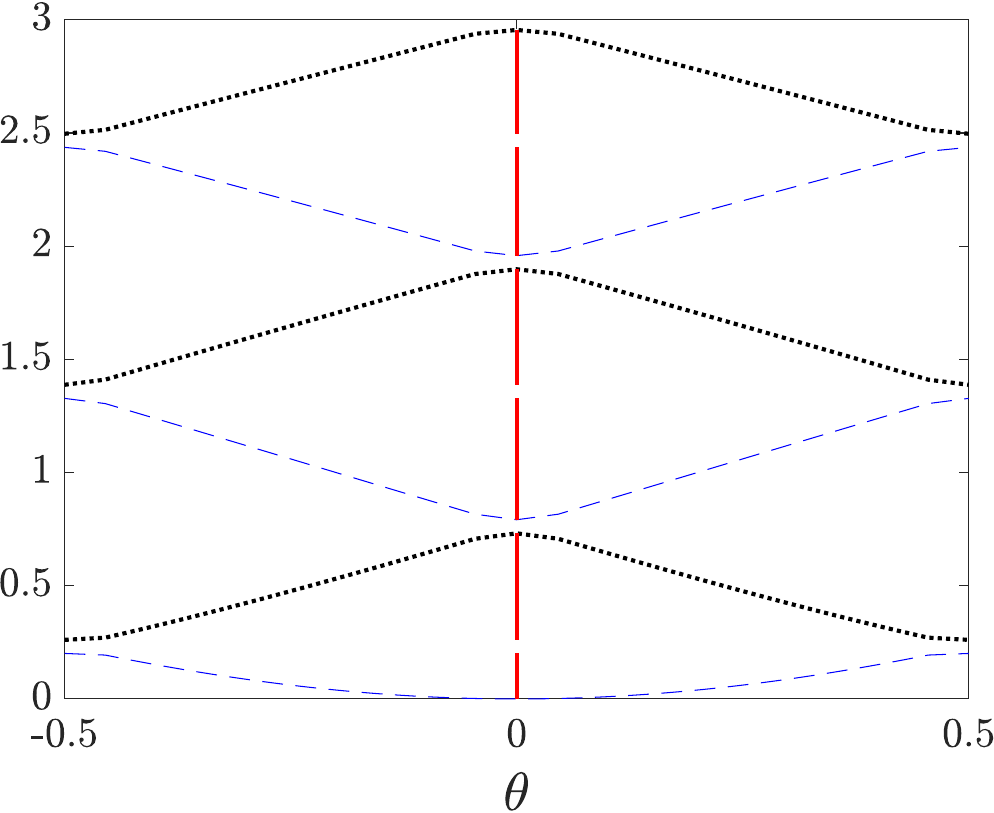}}}
\caption{Sketch of the first six eigenvalues in order of magnitude: (a) flat bottom $\epsilon=0$; (b) in the presence of a small generic bottom perturbation $\epsilon>0$. The dashed blue (resp. dot black) curve represents $\lambda_{2p}^\epsilon(\theta)$ (resp. $\lambda_{2p+1}^\epsilon(\theta)$). The spectra of the operators $G[\epsilon b]$ are represented by the solid red intervals.}
\label{fig}
\end{figure}

The paper is organized as follows. Section~\ref{sec:2} is devoted to the basic properties of the Dirichlet-Neumann operator. We first introduce the Bloch-Floquet transform which allows us to represent any $L^2$-function as the integral over $(-\frac{1}{2},\frac{1}{2}]$ of $\theta$-periodic functions. Following \cite{RS78, K16}, we express $G[\epsilon b]$ as a direct integral decomposition of $G_\theta[\epsilon b]$.
We then write the variational formulation of the elliptic problem associated to $G_\theta[\epsilon b]$ and to its resolvent $(1+G_\theta [\epsilon b])^{-1}$. 
Section~\ref{sec:3} is devoted to general properties of the spectrum of $G[\epsilon b]$ and $G_\theta[\epsilon b]$. An important property of the Dirichlet-Neumann operator is that it is analytic with respect to the bottom \cite{Lannes}. We extend it to the analyticity of its resolvent with respect to $\epsilon$ and $\theta$. This result is central for the description of the spectrum of the operator $G_\theta[\epsilon b]$. In Section~\ref{sec:3.3}, using general properties of perturbation of self-adjoint operators \cite{L22,Rellich}, we show that for $\theta$ not too close to $0,\pm \frac{1}{2}$, the spectrum of $G_{\theta}[\epsilon b]$ is composed of simple eigenvalues that are close to those of $G_{\theta}[0]$ and give estimates on their location. It will be useful later to ensure gaps constructed in Section~\ref{section:4} remain open. We also give a first description of the spectrum of $G_{\theta}[\epsilon b]$ near double eigenvalue of $G_{\theta}[0]$.
We then prove Theorem~\ref{union-of-bands} that describe the spectrun of $G[\epsilon b] $ as unions of bands.
In Section~\ref{section:4}, we show necessary and sufficient conditions for the opening of a gap at $\theta=0,1/2$ of order $\cal O(\epsilon)$. The matching the inner and outer asymptotics on an overlap region leads to the opening of a gap of order $\epsilon$ (Theorem~\ref{one_gap-order-eps}). 
 In particular, assuming that $\widehat{b}_k \ne 0$ for all $|k|<N$ leads to the opening of $N$ gaps. 
In Section~\ref{section:5}, we extend the above analysis to construct gaps of order $\epsilon^2$. 

Our method provides necessary and sufficient conditions on the bottom topography that lead to opening of gaps of order $\epsilon$ and $\epsilon^2$. We believe that a higher order calculation would lead to opening at higher order in $\epsilon$. Because the smallness $\epsilon$ depends on $p$, we are only able to exhibit bottom configurations that lead to the opening of a finite number of gaps. The opening of an infinite number of gaps is an open problem.

We conclude the introduction with some notations:
\begin{align*}
\Omega_\epsilon &:= \{(x,z) \in \R\times \mathbb{R} \; ; \; -1 + \epsilon b(x) < z < 0\}, \\
\omega_\epsilon &:= \{(x,z) \in \T_{2\pi} \times \mathbb{R} \; ; \; -1 + \epsilon b(x) < z < 0\}, \\
S &:= \T_{2\pi}\times[-1,0], \qquad
\gamma := \T_{2\pi} \times\{0\} .
\end{align*}
We denote $\T_{2\pi}:=\R/{2\pi\Z}$ the flat torus of length $2\pi$ and $f\in H^s(\T_{2\pi})$ means that $f\in H^s_{\loc}(\R)$ and is $2\pi$-periodic (namely, $f(x+2\pi)=f(x)$ for a.e. $x\in \R$), whereas a function $F$ belonging in $H^s(\omega_\epsilon)$ or in $H^s(S)$ means that $\Phi$ is periodic in the horizontal direction: $F(x+2\pi,z)=F(x,z)$.

\section{The Bloch-Floquet transform of the Dirichlet-Neumann operator}\label{sec:2}

\subsection{Basic properties of the Dirichlet-Neumann operator}\label{subsec:21}

The goal is to study the Dirichlet-Neumann operator $G[\epsilon b]:=G[0,\epsilon b]$ which naturally appears when solving the linearized water wave equations in the domain $\Omega_\epsilon $, where $b\in C^2(\R)$, bounded, and $\epsilon$ satisfies $\epsilon \|b\|_{L^\infty} <1$. Without loss of generality, we assume $h=1$ and $\int_0^{2\pi} b(x) dx =0$.
Note that this domain is bounded in the vertical direction which allows to have a Poincaré inequality (see \cite[Equation (2.8)]{Lannes} or Lemma~\ref{lem-Poincare}) and to solve the following elliptic problem by the Lax-Milgram theorem.

For any $\xi\in \dot H^1(\R)$, let $\Phi$ be the unique variational solution of
\begin{equation}\label{def-G}
 \begin{cases}
 \Delta \Phi =0 \quad \text{in }\Omega_\epsilon, \\
 \Phi_{\vert_{z=0}} = \xi, \quad \partial_n\Phi_{\vert_{z=-1+\epsilon b}} = 0,
\end{cases}
\end{equation}
see \cite[Proposition~2.9]{Lannes}.
From $\Phi$, we define the Dirichlet-Neumann operator $G[\epsilon b]$ as $G[\epsilon b]\xi = \partial_n\Phi_{\vert_{z=0}}$.

By elliptic regularity, $G[\epsilon b]$ is a continuous operator from $\dot H^1(\R)$ to $L^2(\R)$. 
It is positive semi-definite, symmetric for the $L^2$ scalar product (see \cite[Proposition~3.9]{Lannes}), and it is also self-adjoint on $L^2(\R)$ with domain $H^1(\R)$. This property was shown in \cite{RT11} for flat bottom using symbolic analysis and in \cite[Appendix A.2]{Lannes} for $b\in H^{t_0+1}(\R)$ for $t_0>1/2$. Looking at the details the proof in \cite[Proposition~A.14]{Lannes}, we notice that the decay of $b$ at infinity is not used and that the proposition holds true for periodic $b$ smooth enough, in $C^2$ for instance.

Another well-known property for the Dirichlet-Neumann operator on a Riemannian manifold (\cite[Section 7.11]{Taylor}) and proved in \cite[Corollary 3.6]{RT11} for the fluid domain $\Omega_\epsilon$ is that it is a first-order elliptic operator. This property is not used here but just recalled for sake of completeness.

We conclude this subsection with the standard Poincar\'e inequality when the domain is bounded in one direction. 
Note that important points in the following inequality are that the constant $C_P$ does not depend on $\epsilon$, and the coefficient in front of $\|\partial_z \phi\|_{L^2}$ is strictly smaller than 1. 

\begin{lemma}\label{lem-Poincare} 
There exist $\varepsilon_{0},C_{P}>0$ such that, for any $\varepsilon\in [-\epsilon_0,\varepsilon_{0}]$ and all $\phi\in H^1(\omega_\epsilon)$ 
 \begin{equation} \label{ineq-poincare}
\|\phi\|_{L^2(\omega_\epsilon)} \leq C_{P}\|\phi(\cdot,0)\|_{L^2(0,2\pi)} +\frac34 \|\partial_z \phi\|_{L^2(\omega_\epsilon)}.
 \end{equation}
\end{lemma}

\begin{proof}
For $\phi \in \mathcal{C}^1(\omega_{\varepsilon})$ we write
\begin{equation*}
\phi(x,z) = \phi(x,0) - \int_z^0 \partial_z \phi(x,w)\dd w
\end{equation*}
 thus
\begin{equation*}
|\phi(x,z)| \leq |\phi(x,0)| + |z|^{\frac12}\Big(\int_{-1+\epsilon b(x)}^0 |\partial_z \phi(x,w)|^2\dd w\Big)^{\frac12}=:|\phi(x,0)| + A(x,z).
\end{equation*}
where
\[
\|A\|_{L^2(\omega_\epsilon)}^2
 \leq \frac12 (1+|\epsilon| \|b\|_{L^\infty})^2 \| \partial_z \phi \|_{L^2(\omega_\epsilon)}^2 
\]
which would end the proof if $\phi$ vanishes on the boundary.
For all $\delta>0$, there is $C_{\delta}>0$ such that
\[
|\phi(x,z)|^2 \leq |\phi(x,0)|^2 + 2 |\phi(x,0)| A(x,z)+A(x,z)^2 \leq C_{\delta}|\phi(x,0)|^2 +(1+\delta) A(x,z)^2
\]
hence, 
\[
\|\phi\|_{L^2(\omega_{\varepsilon})} \leq \sqrt{C_{\delta}(1+|\epsilon| \|b\|_{L^\infty})}\|\phi(\cdot,0)\|_{L^2(0,2\pi)}+\sqrt{\frac{1+\delta}2}(1+|\epsilon| \|b\|_{L^\infty}) \| \partial_z \phi \|_{L^2(\omega_{\varepsilon})}.
\]
Choosing $\varepsilon_{0}$ and $\delta$ small enough, leads to \eqref{ineq-poincare}, because $1/\sqrt2 < 3/4$.
\end{proof}

\subsection{Bloch-Floquet transform}\label{sec:Bloch}

The Bloch-Floquet transform, also referred as Gelfand transform, is defined on $\cal S(\R)$ as 
\[
f(x) \mapsto \cal{U} f(x,\theta):= \sum_{n=-\infty}^{\infty} f(x+2\pi n) e^{-2\pi i\theta n}. 
\]
It satisfies
\[
 \cal{U} f(x+2\pi,\theta) = e^{2\pi i\theta}\cal{U} f(x,\theta)
\]
and is uniquely extendable to a unity operator from $L^2(\R)$ in $L^2((-\frac{1}{2},\frac{1}{2}];L^2(0,2\pi))$ by Fubini and Plancherel theorems (see for instance \cite[Page 290]{RS78}). For $f\in \cal S(\R)$,
\[
f(x) = \int_{-1/2}^{1/2} \cal{U} f(x,\theta)\dd \theta \quad \text{ in }\R.
\]
Denoting $g=\cal U f$ with $f\in \cal S(\R)$, there is an explicit formula for $\cal U^\ast g$ only in terms of the values of $g$ in $(-1/2,1/2]\times[0,2\pi]$:
\[
(\cal U^\ast g)(x+2\pi n)=\int_{-1/2}^{1/2} e^{2i \pi n\theta} g(x,\theta)\dd \theta \text{ for all }(x,n)\in [0,2\pi)\times \Z.
\]
This definition of the Bloch-Floquet transform is convenient because it implies that $\cal U^\ast g$ is an isometry from $L^2((-\frac{1}{2},\frac{1}{2}];L^2(0,2\pi))$ in $L^2(\R)$.
The definition of $\cal U^\ast$ for $x$ outside $[0,2\pi)$ comes from the fact that we do not precise the periodicity condition of $g$ with respect to $x$ when we only considered $g$ as an element of $L^2((-\frac{1}{2},\frac{1}{2}];L^2(0,2\pi))$.
An important consequence of the isometry property is that we can decompose any $f\in L^2(\R)$ as an integral of $\theta$-periodic functions (namely, $g(x+2\pi,\theta)=e^{2\pi i\theta} g(x,\theta)$). For more details, we refer to \cite[Section XIII.16]{RS78} and \cite[Section 4.2]{K16}.
Another possible choice of decomposition is based on Fourier transforms as in \cite{APR11}, which is well-adapted to $d\geq 2$.

With our choice of direct integral decomposition of functional spaces
\begin{equation}\label{Floquet-decompspace}
 L^2(\R)=\int_{(-\frac{1}{2},\frac{1}{2}]}^\bigoplus L^2_\theta \dd \theta,
\end{equation}
the goal is to decompose the Dirichlet-Neumann operator $G[\epsilon b]$ into operators $G_\theta[\epsilon b]$ acting, for all $\theta$, on periodic functions.

Let $b\in C^2(\T_{2\pi})$, and for any $\phi\in H^1(\T_{2\pi})$, let $\Phi$ be the unique\footnote{The existence and the uniqueness of $\Phi$ is proved in the proof of Theorem~\ref{theo-def-Gtheta}.} variational solution of
\begin{equation}\label{D2N-theta}
 \begin{cases}
 (-\Delta -2i\theta \partial_x +\theta^2 ) \Phi =0 \quad \text{in }\omega_\epsilon , \\
 \Phi_{\vert_{z=0}} = \phi, \quad (\partial_n+i\theta n_x )\Phi_{\vert_{z=-1+\epsilon b}} = 0,
\end{cases}
\end{equation}
where $n_{x}$ denotes the horizontal component of the outward normal vector $n$.

\begin{theorem}\label{theo-def-Gtheta}
There is $\epsilon_0>0$ such that, for all $\theta\in [- \frac{1}{2}- \epsilon_0,\frac{1}{2}+\epsilon_0]$ and $\epsilon\in [-\varepsilon_{0},\epsilon_0]$, the linear operator $G_\theta[\epsilon b]$ defined as
\begin{equation} \label{DN0-theta-def}
G_\theta[\epsilon b]\phi = \partial_n\Phi_{\vert_{z=0}} \in L^2(\T_{2\pi})
\end{equation}
is well defined on $H^1(\T_{2\pi})$, closed, symmetric, positive semi-definite and bounded uniformly with respect to $\theta\in [-\frac12- \epsilon_0,\frac12+\varepsilon_{0}]$ and $\varepsilon\in [-\varepsilon_{0},\varepsilon_{0}]$.
\end{theorem}

\begin{proof}
It is proved in \cite[Proposition~2.2]{CGLS18} that $G_{\theta}[\epsilon b]$ is well-defined. We briefly recall the argument. For $\phi \in H^1(\T_{2\pi})$ given, we lift the boundary condition on $z=0$ by\footnote{$F$ is linear and continuous from $H^1(\T_{2\pi})$ to $H^{3/2}(\T_{2\pi}\times (-2,0))$ and a possible construction is $F(\phi)(x,z)=\sum_{k\in \Z} \Big(\sum_{j=0}^1 z^j \widehat{\phi^{(j)}}_k h(z\sqrt{1+k^2}) \Big) e^{ikx} $, where $h\in C^\infty_c([0,1),[0,1])$, $h_{\vert_{[0,1/2]}}=1$, see \cite[Theorem~3.1]{LP20} and \cite[Section 2.5]{N12}.} $F(\phi)\in H^{3/2}(\T_{2\pi}\times(-2,0))$ where $F(\phi)_{\vert_{z=0}}=\phi$, $F(\phi)_{\vert_{z=-2}}=0$ and introduce $\Phi=\tilde\Phi+F(\phi)$, where $\tilde\Phi$ is the unique solution in $H^1_{\surf,0}(\omega_\epsilon)$ of the variational formulation
\[
a_{\theta,\epsilon}(\tilde\Phi, \varphi )= - a_{\theta,\epsilon}(F(\phi), \varphi )\quad \forall \varphi\in H^1_{\surf,0}(\omega_\epsilon)
\]
where
\[
a_{\theta,\epsilon}(\Psi, \varphi ):= \int_{\omega_\epsilon}\Big[ \nabla \Psi \cdot \nabla \overline{\varphi} + i\theta( \Psi \partial_{x} \overline{\varphi}-\overline{\varphi} \partial_{x} \Psi) +\theta^2 \Psi \overline{\varphi}\Big] 
\]
and $H^1_{\surf,0}(\omega_\epsilon)$ is the set of functions belonging in $H^1(\omega_\epsilon)$ vanishing at the surface, i.e. on $\{z=0\}$. The existence and uniqueness of $\tilde\Phi$ is a consequence of the Lax-Milgram theorem and the Poincaré inequality (Lemma~\ref{lem-Poincare}), since the coercivity property
\begin{align*}
 a_{\theta,\epsilon}(\Psi, \Psi ) &= \int_{\omega_{\varepsilon}} |\nabla \Psi|^2 -2 \theta \mathrm{Im}(\Psi \partial_{x} \overline{\Psi}) +\theta^2 |\Psi |^2\geq ( \|\nabla \Psi \|_{L^2(\omega_{\varepsilon})} -|\theta| \|\Psi \|_{L^2(\omega_{\varepsilon})})^2\\ 
 &\geq \Big(1 - \frac{3|\theta|}{4}\Big)^2 \|\nabla \Psi \|_{L^2(\omega_{\varepsilon})} ^2 \geq C \|\nabla \Psi \|_{L^2(\omega_{\varepsilon})} ^2 \geq \tilde C \| \Psi \|_{H^1(\omega_{\varepsilon})} ^2
\end{align*}
holds with $C$ and $\tilde C$ independent of $\theta\in [-\frac{1}{2}-\varepsilon_{0},\frac{1}{2}+\varepsilon_{0}]$ and $\varepsilon\in [\epsilon_0,\varepsilon_{0}]$, where $\varepsilon_{0}$ is chosen smaller than $1/2$.

As $\Delta \tilde\Phi \in H^{-1/2}(\T_{2\pi}\times (-1/2,0))$, by elliptic regularity, we have $\tilde\Phi, \Phi \in H^{3/2}(\T_{2\pi}\times (-1/4,0))$, thus the normal trace $\partial_z \Phi_{\vert_{z=0}}$ belongs to $L^2(\mathbb{T}_{2\pi})$. We then have obtained that $G_\theta[\epsilon b]$ is a continuous operator from $H^1(\T_{2\pi})$ to $L^2(\T_{2\pi})$ uniformly with respect to $\theta\in [-\frac{1}{2}-\varepsilon_{0},\frac{1}{2}+\varepsilon_{0}]$ and $\varepsilon\in [\epsilon_0,\varepsilon_{0}]$.

Moreover, for any $\phi,\psi \in H^1(\T_{2\pi})$, let $\Phi,\Psi$ be the solutions associated to \eqref{D2N-theta} respectively. Then, 
\begin{align*}
(G_{\theta}[\varepsilon b]\phi,\psi ) &= \int_{0}^{2\pi} \Big(G_{\theta}[\varepsilon b]\phi\Big) \overline{\psi} 
=\int_{\omega_\epsilon}\Big[ \nabla \Phi \cdot \nabla \overline{\Psi} + i\theta( \Phi \partial_{x} \overline{\Psi}-\overline{\Psi} \partial_{x} \Phi) +\theta^2 \Phi \overline{\Psi}\Big] \\
&=(\phi ,G_{\theta}[\varepsilon b]\psi ),
\end{align*}
which implies that $G_{\theta}[\varepsilon b]$ is a positive semi-definite operator, symmetric for the $L^2$-scalar product. The positivity follows from the coercivity:
\[
(G_{\theta}[\varepsilon b]\phi,\phi )\geq ( \|\nabla \Phi \|_{L^2(\omega_{\varepsilon})} -\theta \|\Phi \|_{L^2(\omega_{\varepsilon})})^2.
\]
$G_\theta$ is also closed: let $(\phi_n, G_\theta[\epsilon b]\phi_n)_n$ be a sequence of the graph of $G_\theta[\epsilon b]$ converging in $L^2(0,2\pi)\times L^2(0,2\pi)$ to $(\phi,g)$, for any test function $\Psi \in C^\infty(\omega_\epsilon)$, we have, for all $n$,
\[
(G_{\theta}[\varepsilon b]\phi_n,\Psi(\cdot,0) )=\int_{\omega_\epsilon}\Big[ \nabla \Phi_{n} \cdot \nabla \overline{\Psi} + i\theta( \Phi_{n} \partial_{x} \overline{\Psi}-\overline{\Psi} \partial_{x} \Phi_{n}) +\theta^2 \Phi_{n} \overline{\Psi}\Big].
\]
From the Poincar\'e inequality, we deduce that $(\Phi_{n})_{n}$ is a bounded sequence in $H^1(\omega_{\varepsilon})$:
\begin{align*}
 \| G_{\theta}[\varepsilon b]\phi_n\|_{L^2}^{1/2} \|\phi_{n}\|_{L^2}^{1/2}&\geq (G_{\theta}[\varepsilon b]\phi_n,\phi_{n} )^{1/2} \geq \|\nabla \Phi_{n} \|_{L^2(\omega_{\varepsilon})} -\theta \|\Phi_{n} \|_{L^2(\omega_{\varepsilon})}\\
 &\geq C_{1} \| \Phi_{n} \|_{H^1(\omega_{\varepsilon})} -C_{2} \| \phi_{n} \|_{L^2(\T_{2\pi})} 
\end{align*}
with constants $C_{1}$ and $C_{2}$ independent of $\theta\in [-\frac12-\varepsilon_{0},\frac12+\varepsilon_{0}]$ and $\varepsilon\in [-\epsilon_0,\varepsilon_{0}]$, where $\varepsilon_{0}$ is chosen smaller than $1/2$. Passing to the limit in the previous equality, $(\Phi_n)_n$ tends in the sense of distributions to $\Phi$, solution of \eqref{D2N-theta}, with $\Phi_{\vert_{z=0}} = \phi$ and $\partial_n\Phi_{\vert_{z=0}} = g$. From $\partial_n \Phi_{\vert_{z=0}} = g\in L^2$, elliptic regularity implies that $\Phi\in H^{3/2}(\omega_\epsilon)$, hence that $\phi \in H^1(\T_{2\pi})=D(G_{\theta}[\varepsilon b])$, which concludes the closure of $G_{\theta}[\varepsilon b]$.
\end{proof}

An important tool for the study of spectral properties is the resolvent operator $(1+G_\theta[\epsilon b])^{-1}$. The next proposition relates the resolvent operator to the trace of the unique solution in $H^1(\omega_\epsilon)$ of an auxiliary elliptic system. The variational formulation of this system was introduced in \cite[Section 4.a]{CNR13}. 

\begin{proposition}\label{proposition_variation_formulation_resolvent}
Let $b\in C^2(\T_{2\pi})$. There is $\epsilon_0>0$ such that,
for all $\theta\in [-\frac{1}{2} -\epsilon_0,\frac12+\epsilon_0]$, $\epsilon\in [-\epsilon_0,\epsilon_0]$ and $\xi \in L^2(\T_{2\pi})$, the system:
\begin{equation}\label{equation_resolvent_omega_epsilon}
 \begin{cases}
 (-\Delta -2i\theta \partial_x +\theta^2 ) \Phi =0 \quad \text{in }\omega_\epsilon , \\
 (\partial_n+1)\Phi_{\vert_{z=0}} = \xi, \quad (\partial_n+i\theta n_x )\Phi_{\vert_{z=-1+\epsilon b}} = 0,
\end{cases}
\end{equation}
has a unique variational solution $\Phi \in H^1(\omega_\epsilon)$ and 
\begin{equation*}
(1+G_\theta[\epsilon b])^{-1}\xi=\Phi_{\vert_{z=0}}.
\end{equation*}
Moreover, $(1+G_\theta[\epsilon b])^{-1}$ is bounded from $L^2(\T_{2\pi})$ to $H^1(\T_{2\pi})$ independently of $\theta\in [-\frac{1}{2} -\epsilon_0,\frac12+\epsilon_0]$ and $\epsilon\in [-\epsilon_0,\epsilon_0]$.
\end{proposition}

\begin{proof}
The function $\Phi$ is a variational solution of \eqref{equation_resolvent_omega_epsilon} if and only if, for any $\Psi \in H^1(\omega_\epsilon)$, 
\begin{equation*}
a^R_{\theta,\epsilon}(\Phi,\Psi) = L(\Psi) := \int_{\T_{2\pi}} \xi \overline{\Psi(\cdot,0)} 
\end{equation*}
where
\begin{equation}\label{definition_a}
a^R_{\theta,\epsilon}(\Phi,\Psi):= \int_{\omega_\epsilon}\Big[ \nabla \Phi \cdot \nabla \overline{\Psi} + i\theta( \Phi \partial_{x} \overline{\Psi}-\overline{\Psi} \partial_{x} \Phi) +\theta^2 \Phi \overline{\Psi}\Big] + \int_\gamma \Phi \overline{\Psi} .
\end{equation}
The operator $L$ is continuous since
\begin{equation*}
|L(\Psi)| \leq \norm{\xi}_{L^2(\T_{2\pi})}\norm{\Psi}_{L^2(\gamma)} \leq C\norm{\xi}_{L^2(\T_{2\pi})}\norm{\Psi}_{H^1(\omega_\epsilon)}
\end{equation*}
and $a^R_{\theta,\epsilon}$ is a continuous sesquilinear form. In addition, it is coercive:
\begin{align*}
a^R_{\theta,\epsilon}(\Phi,\Phi) &\geq (\|\nabla \Phi \|_{L^2(\omega_{\varepsilon})} -\theta \|\Phi \|_{L^2(\omega_{\varepsilon})} )^2+ \|\Phi\|_{L^2(\gamma)}^2\\
&\geq \delta (\|\nabla \Phi \|_{L^2(\omega_{\varepsilon})} -\theta \|\Phi \|_{L^2(\omega_{\varepsilon})} )^2+ \|\Phi\|_{L^2(\gamma)}^2\\
&\geq \frac12 \Big( \sqrt{\delta} \Big| \|\nabla \Phi \|_{L^2(\omega_{\varepsilon})} -\theta \|\Phi \|_{L^2(\omega_{\varepsilon})} \Big|+ \|\Phi\|_{L^2(\gamma)} \Big)^2
\end{align*}
for all $\delta\in (0,1]$, hence by Poincaré inequality (Lemma~\ref{lem-Poincare})
\begin{equation*}
(a^R_{\theta,\epsilon}(\Phi,\Phi))^{1/2} \geq \sqrt{\delta} C_1 \|\nabla \Phi \|_{L^2(\omega_{\varepsilon})} + \Big(\frac1{\sqrt2}-\sqrt\delta C_2 \Big)\|\Phi\|_{L^2(\gamma)}\geq C_3\| \Phi \|_{H^1(\omega_{\varepsilon})}
\end{equation*}
with $C_{1}$, $C_{2}$ and $C_3$ independent of $\theta\in [-\frac12-\varepsilon_{0},\frac12+\varepsilon_{0}]$ and $\varepsilon\in [-\epsilon_0,\varepsilon_{0}]$, where $\delta$ is chosen small enough. 
By Lax-Milgram theorem, there is a unique solution $\Phi \in H^1(\omega_\epsilon)$ of $a^R_{\theta,\epsilon}(\Phi,\cdot) = L$. By elliptic regularity\footnote{When $\Delta \Phi=f \in L^2(\T_{2\pi}\times (-\frac34,0))$ and $\partial_n \Phi = g \in L^2(\T_{2\pi})$, we have $\Phi \in H^{3/2}(\T_{2\pi}\times (-\frac12,0))$: to prove this, we lift the boundary condition by $\tilde F(g)(x,z)=\sum_{k\in \Z} z \widehat{g}_k h(z\sqrt{1+k^2}) e^{ikx}$ so that $\tilde F$ is linear and continuous from $L^2(\T_{2\pi})$ to $H^{3/2}(\T_{2\pi}\times (-\frac34,0))$ and conclude that $\Phi -\tilde F(g)\in H^{3/2}(\T_{2\pi}\times (-\frac12,0))$.}, $\Phi \in H^{3/2}(\T_{2\pi}\times (-\frac12,0))$ and $\phi:=\Phi_{\vert_{z=0}} \in H^1(\T_{2\pi})$. $\Phi$ is also the unique solution of \eqref{D2N-theta}, and we have
\[
G_{\theta}[\epsilon b]\phi = \partial_n \Phi_{\vert_{z=0}} = \xi-\phi,
\]
that is, we have found $\phi\in H^1(\T_{2\pi})$ such that $(1+G_{\theta}[\epsilon b])\phi = \xi$. We have thus proved the surjectivity of $(1+G_{\theta}[\epsilon b])$ from $H^1(\T_{2\pi})$ in $L^2(\T_{2\pi})$.
The injectivity is obvious: for $\phi$ in the kernel of $(1+G_{\theta}[\epsilon b])$, the solution $\Phi$ to \eqref{D2N-theta} is solution of \eqref{equation_resolvent_omega_epsilon} with $\xi\equiv 0$, hence $\Phi\equiv 0$ and $\phi\equiv 0$ by uniqueness in \eqref{equation_resolvent_omega_epsilon}.
This ends the proof of the bijectivity of $(1+G_{\theta}[\epsilon b])$ and 
\[(1+G_\theta[\epsilon b])^{-1}\xi=\phi=\Phi_{\vert_{z=0}}.\]
\end{proof}

\begin{remark}\label{rem-resolvente}
In the previous proof, the presence of $\|\Phi\|_{L^2(\gamma)}$ is important to obtain the coercivity, but choosing $\delta$ possibly smaller, we can easily prove that $\lambda\notin \sigma(G_{\theta}[\epsilon b])$ for any $\lambda\in ]-\infty,0)$, only adding $-\lambda$ in front of $\int_\gamma \Phi \overline{\Psi}$ in the definition of $a^R_{\theta,\epsilon}$.
\end{remark}

An important consequence of Proposition~\ref{proposition_variation_formulation_resolvent} is the self-adjointness of $G_{\theta}[\epsilon b]$ and properties of its spectrum.

\begin{corollary}
Let $b\in C^2(\T_{2\pi})$. There is $\epsilon_0>0$ such that, for all $\theta\in [-\frac{1}{2}- \epsilon_0,\frac{1}{2} +\epsilon_0]$ and $\epsilon\in [-\epsilon_0,\epsilon_0]$, the operator $G_\theta[\epsilon b]$ is self-adjoint with domain $H^1(\T_{2\pi})$ and its spectrum $\sigma(G_{\theta}[\epsilon b]) \subset[0,+\infty[$.
\end{corollary}
\begin{proof}
This result comes directly from the classical theorem for closed symmetric operators on Hilbert spaces. Indeed \cite[Theorem~X.1]{RS75} states that the spectrum of $G_{\theta}[\epsilon b]$ is either the closed upper half-plane, the closed lower half-plane, the entire plane or a subset of the real axis.
In the proof of Proposition~\ref{proposition_variation_formulation_resolvent}, we obtained that $(1+G_\theta[\epsilon b])^{-1}$ is bounded from $L^2(\T_{2\pi})$ to $H^1(\T_{2\pi})$, which implies that $-1\notin \sigma(G_{\theta}[\epsilon b])$. Thus the spectrum of $G_{\theta}[\epsilon b]$ is a subset of the real axis.
The third statement in \cite[Theorem~X.1]{RS75} claims, that in this case, the operator is also self-adjoint. Remark~\ref{rem-resolvente} implies that $\sigma(G_{\theta}[\epsilon b]) \subset[0,+\infty[$.
\end{proof}

\begin{remark}\label{rem-eigen-resolvent}
Since its resolvent is compact, the self-adjointness of $G_\theta[\epsilon b]$ implies that it has purely discrete spectrum. There exists an orthonormal basis $(\psi_n(\theta,\epsilon,\cdot))_{n\geq 0}$ of $L^2(\T_{2\pi})$ composed of eigenvectors of $G_{\theta}[\epsilon b]$, where the eigenvalues $(\lambda_n(\theta,\epsilon))_{n\geq 0}$ of $G_{\theta}[\epsilon b]$ are real numbers that we can order such that $(\lambda_n)_n$ is increasing and tends to $+\infty$ as $n$ tends to infinity. 
Their multiplicity is finite and $\sigma (G_{\theta}[\epsilon b])=\{\lambda_n(\theta,\epsilon),n\in \N\}$. Note that $G_{\theta}[\epsilon b]\psi_n(\theta,\epsilon,\cdot) = \lambda_n(\theta,\epsilon)\psi_n(\theta,\epsilon,\cdot)$ implies that $\Phi_n$, solution of \eqref{D2N-theta} with $\phi=\psi_n$, is also solution of \eqref{equation_resolvent_omega_epsilon} with $\xi=(1+\lambda_n)\psi_n$ and that $(1+G_\theta[\epsilon b])^{-1}) (1+\lambda_n)\psi_n = \psi_n$, which means that $\psi_n$ is an eigenfunction of the resolvent with eigenvalue $\tau_n(\theta,\epsilon)=1/(1+\lambda_n(\theta,\epsilon))$.
\end{remark}

A second consequence of the self-adjointness is that the definition of $G_\theta[\epsilon b]$ given in Theorem~\ref{theo-def-Gtheta} provides the appropriate integral decomposition of $G[\epsilon b]$ as expressed in the next theorem.

\begin{theorem}\label{theo:Floquet-decompop}
Under the decomposition \eqref{Floquet-decompspace}, we have
\begin{equation}\label{Floquet-decompop}
 \cal U G[\epsilon b] \cal U^\ast = \int_{(-\frac{1}{2},\frac{1}{2}]}^\bigoplus e^{i\theta x} G_\theta[\epsilon b] e^{-i\theta x} \dd \theta .
\end{equation}
\end{theorem}

\begin{proof}
We follow the proof of \cite[Equation (148) page 289]{RS78}.
Denote $A$ the operator in the right hand side of \eqref{Floquet-decompop}. Since $G_\theta[\epsilon b]$ is self-adjoint for all $\theta \in (-\frac12,\frac12]$, it follows from \cite[Theorem XIII.85 (a)]{RS78} that $A$ is self-adjoint. Since $G[\epsilon b]$ is also self-adjoint (see Subsection~\ref{subsec:21}) and since a symmetric operator can at most have one self-adjoint extension it is sufficient to show that if $\xi\in \cal S(\R)$, then $\cal U \xi \in D(A)$ and $\cal U G[\epsilon b]\xi = A\cal U\xi $. 

For $\xi\in \cal S(\R)$, from the definition of $\cal U$ as a convergent sum, $\cal U\xi\in C^\infty$. For every fixed $\theta\in (-\frac12,\frac12]$, we have $e^{-i\theta (x+2\pi)}\cal U\xi(x+2\pi,\theta) = e^{-i\theta x} \cal U \xi(x,\theta)$, hence $x\mapsto e^{-i\theta x} \cal U \xi(x,\theta)$ belongs to $H^1(\T_{2\pi})=D(G_\theta[\epsilon b])$, which implies that $\cal U \xi \in D(A)$.

Next, we use the definition of $G_\theta[\epsilon b]$ to consider $\Phi_\theta$ solution of \eqref{D2N-theta} for $\phi(x) = e^{-i\theta x} \cal U \xi(x,\theta)$, then $A\cal U\xi(\theta,x)=e^{i\theta x} \partial_n \Phi_\theta {\vert_{z=0}}(x)$.
On the other hand, let $\Phi$ be the solution of \eqref{def-G}, $G[\epsilon b]\xi =\partial_n \Phi_{\vert_{z=0}}$, and 
\[
\cal U G[\epsilon b]\xi(x,\theta)=\sum_{n=-\infty}^{\infty} (G[\epsilon b]\xi)(x+2\pi n) e^{-2\pi i\theta n} = e^{i\theta x} \partial_n \tilde\Phi_\theta{\vert_{z=0}}(x)
\]
where
\[
\tilde\Phi_\theta(x,z) =e^{-i\theta x} \sum_{n=-\infty}^{\infty} \Phi(x+2\pi n, z) e^{-2\pi i\theta n}.
\]
Noticing that $\tilde\Phi_\theta$ is solution of \eqref{D2N-theta} for $\phi(x) = e^{-i\theta x} \cal U \xi(x,\theta)$ allows to conclude that $\tilde\Phi_\theta=\Phi_\theta$, which ends the proof.
\end{proof}

We conclude this section by noticing that the spectrum of $G_\theta[\epsilon b]$ is even with respect to $\theta$. Indeed, for any $\phi \in H^1(\T_{2\pi})$, taking the conjugate of the elliptic problem \eqref{D2N-theta} associated to $G_\theta[\epsilon b]\phi$, we observe that $\overline{\Phi}$ is the solution related to $G_{-\theta}[\epsilon b]\overline{\phi}$, hence
\[
G_{-\theta}[\epsilon b]\overline{\phi} = \overline{G_{\theta}[\epsilon b] \phi}.
\]
An eigenpair $(\lambda_n,\phi_n)$ of $G_\theta[\epsilon b]$ gives rise to an eigenpair $(\lambda_n,\overline{\phi_n})$ of $G_{-\theta}[\epsilon b]$, which implies the evenness of $\lambda_n$. It is thus sufficient to restrict the study of the eigenvalues to $\theta\in[0,1/2]$.

\section{Analyticity and general properties of the spectrum of $G[\epsilon b]$}\label{sec:3}

\subsection{Flat bottom}\label{sec:fond-plat}
When the bottom is flat, $\epsilon=0$, the eigenvalues of $G_\theta[0]$ are 
\begin{equation*} 
 \kappa_p(\theta) = (p+\theta)\tanh(p+\theta)
\end{equation*}
for $p \in \Z$ and Bloch parameter $\theta \in (-\frac12,\frac12]$. 
The associated eigenfunctions are $e^{ipx}$, where the solution of the elliptic problem \eqref{D2N-theta} (for $\epsilon=0$ and $\phi=e^{ipx}$) is 
\begin{equation} \label{phi-p}
\Phi_p(\theta,x,z)=e^{ip x}\frac{\cosh ((p+\theta)(z+1))}{\cosh (p+\theta)} .
\end{equation}
Eigenvalues are simple for $-1/2 < \theta < 0$ and 
$0 < \theta < 1/2$. For $ \theta = 0, 1/2$, the eigenvalues $\kappa_p(\theta)$ have multiplicity two.

When reordered appropriately by their size, the eigenvalues and eigenfunctions of $G_\theta[0]$ are given as (See Figure~\ref{fig}.a):
\begin{align*}
 & \textrm{For} \ & -\tfrac{1}{2} \le \theta < 0, \quad 
 & \lambda_{2p}^{0} (\theta) = \kappa_{-p}(\theta);
 \ &\psi_{2p}(\theta,x) = \tfrac1{\sqrt{2\pi}} e^{-ipx} , \\
 & \textrm{for} \ & 0 \le \theta \le \tfrac{1}{2}, \quad 
 & \lambda_{2p}^{0} (\theta) = \kappa_{p}(\theta);
 \ &\psi_{2p}(\theta,x) = \tfrac1{\sqrt{2\pi}}e^{ipx} , 
\end{align*}
and
\begin{align*}
 & \textrm{for} \ & -\tfrac{1}{2} \le \theta < 0, \quad 
 & \lambda_{2p-1}^{0} (\theta) = \kappa_{p}(\theta); \ &\psi_{2p-1}(\theta,x) = \tfrac1{\sqrt{2\pi}}e^{ipx}, \\
 & \textrm{for} \ & 0 \le \theta \le \tfrac{1}{2}, \quad 
 & \lambda_{2p-1}^{0} (\theta) = \kappa_{-p}(\theta);
 \ &\psi_{2p-1}(\theta,x) = \tfrac1{\sqrt{2\pi}}e^{-ipx} .
\end{align*}
The eigenvalues are continuous functions of the Bloch parameter $\theta\in (-1/2,1/2]$.

As explained in Remark~\ref{rem-eigen-resolvent}, the eigenvalues of the resolvent $(1+G_\theta[0])^{-1}$ are 
\[
\tau_p^{0}(\theta) = (1+\lambda_p^{0}(\theta) )^{-1}
\]
with the same eigenfunctions.
As it will be needed in Sections~\ref{section:4} and \ref{section:5} we conclude this section by discussing the application $R_{p,\theta}$ defined on $L^2(\T_{2\pi})$ as
\begin{equation*}
R_{p,\theta} f :=\Big((1+G_\theta[0])^{-1} - \tau_{p}^{0}(\theta)\Big) f.
\end{equation*}
For $\theta =0$ and any $p\in \N^*$, $\tau_{2p}^{0}(0)$ is of multiplicity two, with the eigenfunctions $\psi_{2p}^{0}(0,x)=(2\pi)^{-1/2}e^{ipx}$ and $\psi_{2p-1}^{0}(0,x)=(2\pi)^{-1/2}e^{-ipx}$. Therefore, for any $f\in L^2(\T_{2\pi})$, we have
\begin{equation*}
R_{2p,0} f = \sum_{k=0}^{+\infty} ( \tau_{k}^{0}(0)- \tau_{2p}^{0}(0) ) ( f , \psi_{k}^{0}(0,x) ) \psi_{k}^{0}(0,x),
\end{equation*}
that is, for $g\in H^1(\T_{2\pi})$, the equation
\[
R_{2p,0} f = g 
\]
has a solution if and only if 
\[
g \in E_{\lambda_{2p}^0(0)}^\bot := \vecspan(\psi_{2p}^{0}(0,\cdot), \psi_{2p-1}^{0}(0,\cdot))^\bot \subset L^2(\T_{2\pi})
\]
i.e., if and only if 
\begin{equation}\label{ortho-0}
\int_{\T_{2\pi}} g(x) e^{-ipx}\dd x=\int_{\T_{2\pi}} g(x) e^{ipx}\dd x =0.
\end{equation}
In other words, $R_{2p,0}$ induces an automorphism on $E_{\lambda_{2p}^0(0)}^\bot$ with an inverse defined by
\begin{equation}\label{operateur-R_pinv}
R_{2p,0}^{-1}g := \sum_{k=0, k\neq 2p,2p-1}^{+\infty}\frac{( g, \psi_{k}^{0}(0,x) )}{\tau_{k}^{0}(0)- \tau_{2p}^{0}(0) }\psi_{k}^{0}(0,x).
\end{equation}
$R_{2p,0}^{-1}$ is then a bounded operator from $E_{\lambda_{2p}^0(0)}^\bot$ to $E_{\lambda_{2p}^0(0)}^\bot \cap H^1(\T_{2\pi})$. We note also that if $\Phi$ is solution of the Laplace problem associated to $G_0[0]f$ (i.e. such that $\Phi_{\vert_{z=0}} =f$) then
\[
\| \Phi \|_{H^{3/2}(S)} \leq C \|f\|_{H^1(\T_{2\pi})} \leq C_p \|g\|_{L^2(\T_{2\pi})}.
\]

Similarly, for $p\in \N$, $R_{2p,\frac12} f = g $ has a solution if and only if
\[
\int_{\T_{2\pi}} g(x) e^{-ipx}\dd x=\int_{\T_{2\pi}} g(x) e^{i(p+1)x}\dd x =0,
\]
which allows also us to construct $R_{2p,\tfrac12}^{-1}$ on $E_{\lambda_{2p}(\frac12,0)}$.

\subsection{Analyticity of the resolvent of $G_{\theta}[\epsilon b]$ in $\epsilon$ and $\theta$}\label{sec:3.2}

It is known that the Dirichlet-Neumann operator $b\mapsto G[b]$ is analytic with respect to the shape of the bottom (see \cite[Appendix A]{Lannes}). These results do not directly apply since we want to keep track of the dependence in the Bloch parameter. The structure of the forthcoming proof is similar to what is done in \cite[Appendix A]{Lannes} but the problem is much simpler since we study analyticity with respect to a real parameter $\epsilon$ instead of studying the dependence with respect to the whole bottom function.

The first step to study the behavior of $G_\theta[\epsilon b]$ with respect to $\epsilon$ is to straighten the fluid domain in order to see explicitly the $\epsilon$-dependence. We choose one of the simplest way to straighten $\omega_\epsilon$ to $S=\T_{2\pi}\times(-1,0)$: Let $\Sigma : S \longrightarrow \omega_\epsilon$ be the diffeomorphism defined by $\Sigma(x,z) = (x,z-\epsilon z b(x))$.

\begin{proposition}\label{proposition_formulation_resolvant_S}
Let $b\in C^2(\T_{2\pi})$ and $\epsilon_0>0$ given in Proposition~\ref{proposition_variation_formulation_resolvent}. For all $\theta\in [-\frac12-\epsilon_0,\frac12+\epsilon_0]$, $\epsilon\in [-\epsilon_0,\epsilon_0]$ and $\xi \in L^2(\T_{2\pi})$, then the function $\tilde \Phi(x,z) = \Phi(\Sigma(x,z)) \in H^1(S)$ is a solution of
\begin{equation}\label{equation_resolvant_S}
\left\{
\begin{aligned}
& -\div (P(\Sigma)\nabla\tilde \Phi ) - 2i\theta\Bigg(e_1 +\epsilon
\begin{pmatrix} -b(x) \\ z b'(x) \end{pmatrix}
\Bigg)\cdot \nabla \tilde \Phi + \theta^2(1-\epsilon b(x))\tilde \Phi = 0 \quad \text{in } S, \\
& (P(\Sigma)\nabla \tilde \Phi)\cdot e_z - i\theta \epsilon b'(x) \tilde \Phi = 0 \quad \text{on } \{z = -1\} ,\\
& (P(\Sigma)\nabla \tilde \Phi)\cdot e_z + \tilde \Phi = \xi \quad \text{on } \{z = 0\},
\end{aligned}\right.
\end{equation}
if and only if $\Phi$ is a solution of \eqref{equation_resolvent_omega_epsilon}, where 
\begin{equation*}
P(\Sigma) = I_2 + Q(\Sigma) \quad \text{and} \quad Q(\Sigma)= \epsilon
\begin{pmatrix}
 -b(x) & zb'(x) \\ zb'(x) & \frac{b(x)+\epsilon(zb'(x))^2}{1-\epsilon b(x)}
\end{pmatrix}.
\end{equation*}
Moreover, we have 
\begin{equation*}
(1+G_\theta[\epsilon b])^{-1}\xi=\tilde \Phi_{\vert_{z=0}}.
\end{equation*}
\end{proposition}

The proof is a little long and we leave it to the reader. It consists in replacing $\tilde \Phi$ by $\Phi(\Sigma)$ in \eqref{equation_resolvant_S} and inserting the expression of $D\Sigma$ and $P(\Sigma)$. For more details about the straightening, we refer to \cite[Section 2.2.3, page 46]{Lannes}.

\begin{remark}\label{remark_variational_formulation_S}
The unique solution $\tilde \Phi$ of \eqref{equation_resolvant_S} is given by
\begin{equation*}
 a^{R,S}_{\theta,\epsilon}(\tilde \Phi,\tilde \Psi) = L(\tilde\Psi) := \int_{\gamma} \xi \overline{\tilde\Psi(\cdot,0)} \quad \text{for all} \; \tilde\Psi \in H^1(S) 
\end{equation*}
where $a^{R,S}_{\theta,\epsilon}(\Phi(\Sigma),\Psi(\Sigma)) = a^R_{\theta,\epsilon}( \Phi,\Psi)$, with $a^R_{\theta,\epsilon}$ defined in the proof of Proposition~\ref{proposition_variation_formulation_resolvent}.
After the change of variable, we obtain the natural sesquilinear form associated to \eqref{equation_resolvant_S}:
\begin{equation}\label{definition_a_S}
\begin{aligned}
a^{R,S}_{\theta,\epsilon}(\tilde\Phi,\tilde\Psi) 
=& \int_S \Big[P(\Sigma)\nabla \tilde \Phi \cdot \nabla \overline{\tilde \Psi} +
i\theta \Big(e_1 +\epsilon\begin{pmatrix}-b(x) \\z b'(x)\end{pmatrix}\Big) 
\cdot( \tilde \Phi \nabla \overline{\tilde \Psi}-\overline{\tilde \Psi} \nabla \tilde \Phi)\Big]\\
& + \int_S \theta^2 (1-\epsilon b)\tilde\Phi \overline{\tilde\Psi} + \int_\gamma \tilde\Phi \overline{\tilde\Psi} 
\end{aligned}
\end{equation}
From the proof of Proposition~\ref{proposition_variation_formulation_resolvent}, we can state the uniform coercivity of $a^{R,S}$: there exists $\epsilon_0$ and $C$ such that for all $\theta \in [-\frac{1}{2}-\epsilon_0,\frac{1}{2}+\epsilon_0]$ and $\epsilon \in [-\epsilon_0,\epsilon_0]$,
\begin{equation*}
a^{R,S}_{\theta,\epsilon}(U,U) \geq C\norm{U}^2_{H^1(S)}. 
\end{equation*}
\end{remark}

The main advantage of \eqref{equation_resolvant_S} is to work on a fixed domain $S=\T_{2\pi}\times(-1,0)$ and to identify the influences of $\theta$ and $\epsilon$. 
For instance, we adapt in the following remark the end of Remark~\ref{rem-eigen-resolvent} with an elliptic problem satisfied in $S$.

\begin{remark}
If $(\lambda_n,\psi_n)$ is an eigenpair of $G_{\theta}[\epsilon b]$, then $(\tau_n,\psi_n)$ is an eigenpair of the resolvent operator, and the $\Phi_n$, solution of \eqref{D2N-theta} with $\phi=\psi_n$, is the solution of \eqref{equation_resolvent_omega_epsilon} with $\xi=(1+\lambda_n)\psi_n$. This gives rise to $\tilde\Phi_n = \Phi_n(\Sigma)\in H^{3/2}(S)$, solution of \eqref{equation_resolvant_S} with $\xi=(1+\lambda_n)\psi_n$ and such that $\tilde\Phi_n\vert_{z=0}=\psi_n$. To summarize, an eigenfunction of $G_\theta[\epsilon b]$ is the trace on $\gamma$ of a function $\Phi\in H^{3/2}(S)$ satisfying 
\begin{equation}\label{equation_eigenvalue_problem_S1}
\left\{
\begin{aligned}
& -\div (P(\Sigma)\nabla \Phi ) - 2i\theta\Bigg(e_1 +\epsilon
\begin{pmatrix} -b(x) \\ z b'(x) \end{pmatrix}
\Bigg)\cdot \nabla \Phi + \theta^2(1-\epsilon b)\Phi = 0 \quad \text{in } S, \\
& (P(\Sigma)\nabla \Phi)\cdot e_z - i\theta \epsilon b' \Phi = 0 \quad \text{on } \{z = -1\}, \\
& (P(\Sigma)\nabla \Phi)\cdot e_z = \lambda(\theta,\epsilon)\Phi \quad \text{on } \{z = 0\},
\end{aligned}\right.
\end{equation}
where $\lambda(\theta,\epsilon)$ is its associated eigenvalue.
From Remark~\ref{remark_variational_formulation_S}, the above system for $(\Phi, \lambda)$ is equivalent to
\begin{equation}\label{equation_eigenvalue_problem_S2}
a^{R,S}_{\theta,\epsilon}(\Phi,V) = (1+\lambda(\theta,\epsilon)) \int_{\gamma} \Phi \overline{V} \quad \text{for all} \; V \in H^1(S).
\end{equation}
A detailed study of this system will lead in Sections~\ref{section:4} and \ref{section:5} to the construction of approximate eigenvalues of the resolvent operator for $\theta$ close to $0$ and $\frac{1}{2}$. 
\end{remark}

The explicit dependence on $(\theta,\epsilon)$ of the resolvent operator also allows us to prove the analyticity of the resolvent, with respect to $\varepsilon$, uniformly in $\theta\in [-\frac12-\varepsilon_{0},\frac12+\varepsilon_{0}]$.

\begin{proposition}\label{proposition_analyticity_in_epsilon}
There exist $C_0, r > 0$ depending only on $\norm{b}_{W^{1,\infty}}$ such that 
\begin{equation*}
\epsilon \in (-r,r) \mapsto (1+ G_\theta [\epsilon b])^{-1} \in \cal{L}(L^2(\T_{2\pi}); H^1(\T_{2\pi})) 
\end{equation*}
is analytic. More precisely there exists bounded operators $R_k(\theta) \in \cal{L}(L^2(\T_{2\pi}); H^1(\T_{2\pi}))$ such that 
\begin{equation*}
\| R_k(\theta) \xi \|_{H^1(\T_{2\pi})} \leq C_0 \|\xi \|_{L^2(\T_{2\pi})}r^{-k}, \quad \forall \xi \in L^2(\T_{2\pi}),\ \theta\in [-\tfrac12-\varepsilon_{0},\tfrac12+\varepsilon_{0}] \text{ and }k\in \N 
\end{equation*}
and
\begin{equation*}
(1+ G_\theta [\epsilon b])^{-1} = \sum_{k=0}^{+\infty} \epsilon^k R_k(\theta)
\end{equation*}
where the series converges in $\cal{L}(L^2(\T_{2\pi}); H^1(\T_{2\pi}))$.
\end{proposition}

\begin{proof}
Let us fix $\xi \in L^2(\T_{2\pi})$ and $\Phi$ the associated solution of \eqref{equation_resolvant_S}. We write the expansion of $P(\Sigma)$ in terms of $\epsilon$:
\[
P(\Sigma) = I_2 + \sum_{k=1}^{+\infty} \epsilon^k Q_k 
\]
 with 
\begin{equation}\label{Q1-Qk}
Q_1 = 
\begin{pmatrix}
 -b(x) & zb'(x) \\ zb'(x) & b(x)
\end{pmatrix}
\text{ and }
Q_k = 
\begin{pmatrix}
 0 & 0 \\ 0 & b^k(x) + (zb'(x))^2b^{k-2}(x)
\end{pmatrix} \text{ for } k \geq 2.
\end{equation}
Including this expression in \eqref{equation_resolvant_S}, we get that $\Phi$ solves:
\begin{equation*}
\left\{
\begin{aligned}
& -\Delta\Phi -2i\theta \partial_x \Phi + \theta^2 \Phi \\
&- \epsilon 2i\theta \begin{pmatrix}
-b(x) \\
z b'(x)
\end{pmatrix}\cdot \nabla \Phi - \epsilon \theta^2 b(x) \Phi 
- \sum_{k=1}^{+\infty}\epsilon^k\div(Q_k\nabla \Phi) = 0 \quad \text{in } S \\
&\partial_z \Phi -\epsilon i\theta b'(x) \Phi +\sum_{k=1}^{+\infty}\epsilon^k(Q_k \nabla \Phi)\cdot e_z = 0 \quad \text{on } \{z = -1\} \\
& \partial_z \Phi+ \Phi +\sum_{k=1}^{+\infty}\epsilon^k(Q_k \nabla \Phi)\cdot e_z = \xi \quad \text{on } \{z = 0\}.
\end{aligned}\right.
\end{equation*}
Plugging inside an expansion of $\Phi = \sum_{k=0}^{+\infty} \epsilon^k \Phi_k$, we identify the terms of order 1 to write
\begin{equation*}
\left\{
\begin{aligned}
& -\Delta\Phi_0 -2i\theta \partial_x \Phi_0 + \theta^2 \Phi_0 = 0 \quad \text{in } S \\
& \partial_z \Phi_0 = 0 \quad \text{on } \{z = -1\} \\
& \partial_z \Phi_0 + \Phi_0 = \xi \quad \text{on } \{z = 0\}.
\end{aligned}\right.
\end{equation*}
For terms of order $\epsilon^k$ with $k \geq 1$ we obtain
\begin{equation}\label{equation_phi_k}
\left\{
\begin{aligned}
& -\Delta\Phi_k -2i\theta \partial_x \Phi_k + \theta^2 \Phi_k \\
&=
2i\theta\begin{pmatrix}
-b(x) \\
z b'(x)
\end{pmatrix}\cdot \nabla \Phi_{k-1} + \theta^2 b(x) \Phi_{k-1} 
+ \sum_{j=1}^k\div(Q_j\nabla \Phi_{k-j}) \quad \text{on } S \\
& \partial_z \Phi_k = i\theta b'(x) \Phi_{k-1} -\sum_{j=1}^k(Q_j \nabla \Phi_{k-j})\cdot e_z \quad \text{on } \{z = -1\} \\
& \partial_z \Phi_k + \Phi_k = - \sum_{j=1}^k(Q_j \nabla \Phi_{k-j})\cdot e_z \quad \text{on } \{z = 0\}.
\end{aligned}\right.
\end{equation}
These systems correspond to the elliptic problem associated to $(1+G_\theta[0])^{-1}$, and as in Proposition~\ref{proposition_variation_formulation_resolvent}, we identify the variational formulation 
\begin{equation*}
a^R_{\theta,0}(\Phi,\Psi) = L_k(\Psi) 
\end{equation*}
where $a^R_{\theta,\epsilon}$ is defined in \eqref{definition_a}, $L_0(\Psi) = \int_S \xi\overline{\Psi(\cdot,0)}$, and for $k \geq 1$,
\begin{equation*}
L_k(\Psi) 
= i\theta \int_S \begin{pmatrix}-b(x) \\z b'(x)\end{pmatrix}\cdot(\nabla \Phi_{k-1} \overline{\Psi} - \Phi_{k-1}\nabla\overline{\Psi}) 
+ \theta^2 \int_S b \Phi_{k-1}\overline{\Psi} \\
- \sum_{j=1}^k \int_S Q_j \nabla \Phi_{k-j}\cdot \nabla \overline{\Psi}.
\end{equation*}
From Proposition~\ref{proposition_variation_formulation_resolvent}, these systems have a unique solution in $H^1(S)$. By elliptic regularity,
\[
\| \Phi_0 \|_{H^{3/2}(S)} \leq C_1 \| \xi\|_{L^2(\T_{2\pi})}
\]
and, denoting $F_{k}$ the term in the right-hand side of the first equation of \eqref{equation_phi_k},
\[
\| \Phi_k \|_{H^{3/2}(S)} \leq C \| F_k \|_{H^{-1/2}(S)}
\leq C_1\sum_{j=1}^k \|b\|_{W^{1,\infty}}^j \| \Phi_{k-j} \|_{H^{3/2}(S)}
\]
where $C_1$ is independent of $\theta \in [-\frac{1}{2}-\epsilon_0,\frac{1}{2}+\epsilon_0]$.

Setting $r=\min(1/(2\|b\|_{W^{1,\infty}}) ; 1/(2C_{1}\|b\|_{W^{1,\infty}})) $, one proves by induction that
\[
\| \Phi_k \|_{H^{3/2}(S)} \leq C_{1} r^{-k} \| \xi\|_{L^2(\T_{2\pi})} \quad \forall k\geq 0.
\]
Setting $R_{k}(\theta) :\ \xi\mapsto \Phi_k(\cdot,0)$ ends the proof.
\end{proof}

\begin{remark}
 As expected, we note in the previous proof that $R_{0}=(1+G_{\theta}[0])^{-1}$.
\end{remark}

Following the strategy of the previous proof and writing $\theta^2 = \theta_0^2 + 2\theta_0(\theta-\theta_0) + (\theta-\theta_0)^2$, we may also prove the analyticity with respect to $\theta$ uniformly in $\epsilon$.

\begin{proposition}\label{proposition_analyticity_in_theta}
There exists $C_0, r > 0$ depending only on $\| b\|_{W^{1,\infty}}$ such that 
\[
\theta \in (-\tfrac12 -r,\tfrac12+r) \mapsto (1+ G_\theta [\epsilon b])^{-1} \in \cal{L}(L^2(\T_{2\pi}); H^1(\T_{2\pi})) 
\]
is analytic. More precisely, there exist bounded operators $\tilde R_k[\theta_0,\epsilon] \in \cal{L}(L^2(\T_{2\pi}); H^1(\T_{2\pi}))$ such that 
\begin{equation*}
\| \tilde R_k[\theta_0,\epsilon] \xi \|_{H^1(\T_{2\pi})} \leq C_0 \|\xi \|_{L^2(\T_{2\pi})}r^{-k}, \quad \forall \xi \in L^2(\T_{2\pi}),\ \forall k\in \N ,
\end{equation*}
for all $\theta_0\in [-\frac12,\frac12]$, $\epsilon\in [-\epsilon_0,\epsilon_0]$, and
\begin{equation*}
(1+ G_\theta [\epsilon b])^{-1} = \sum_{k=0}^{+\infty} \tilde R_k[\theta_0,\epsilon](\theta-\theta_0)^k
\end{equation*}
where the series converges in $\cal{L}(L^2(\T_{2\pi}); H^1(\T_{2\pi}))$.
\end{proposition}

\subsection{Perturbation theory of self-adjoint operators}\label{sec:3.3}

The main tool of this section is perturbation theory of the spectrum of self-adjoint operators. We start with a general result: let $A$ be a self-adjoint operator on its domain $D(A) \subset H$, where $H$ is a separable Hilbert space, and $a_1, a_2$ real numbers in the resolvent set $\rho(A)$, with $a_1<a_2$. Let $(B_\epsilon)_{\epsilon\in (0,\epsilon_0]}$ be a family of symmetric operators on $D(A)$ such that $B_\epsilon (A+i)^{-1}$ is bounded uniformly with respect to $\epsilon$.
Set 
\begin{equation*}
 \epsilon_A = \inf_{\epsilon\in (0,\epsilon_0]} \Bigg( \frac{1}{2\|B_\epsilon(A-a_1)^{-1} \|} ;
 \frac{1}{2\|B_\epsilon(A-a_2)^{-1} \| }; \epsilon_0
 \Bigg),
\end{equation*}
which is well-defined from the hypotheses above. Indeed, by \cite[Corollary 4.6]{L22}, we have
\begin{equation*}
 \mathrm{d}(a_1, \sigma(A)) =\min_{s\in\sigma(A)} |s-a_1| = \|(A-a_1)^{-1}\|^{-1},
\end{equation*}
which implies that
\begin{equation}\label{est-BA}
\|B_\epsilon(A-a_1)^{-1} \| \le \|B_\epsilon(A+i)^{-1}\| \Big\|\frac{A+i}{A-a_1}\Big\| \le 
\|B_\epsilon(A+i)^{-1}\| 
\Big( 1+ \frac{|a_1| +1}{ \mathrm{d}(a_1, \sigma(A))} \Big).
\end{equation}

\begin{theorem} \label{th-Le}
If $\epsilon \in (-\epsilon_0,\epsilon_0) \mapsto B_\epsilon$ is analytic in $\cal{L}(D(A),H)$ (in the sense of Proposition~\ref{proposition_analyticity_in_epsilon}), then for all $|\epsilon|<\epsilon_A$, the perturbed operator $A+ \epsilon B_\epsilon$ has the following properties:
\begin{enumerate}[(a)]
 \item $a_1$ and $a_2$ are in the resolvent set of 
 $A+\epsilon B_\epsilon $ for all $|\epsilon|<\epsilon_A$.
 \item If $\sigma(A)\cap [a_1,a_2]$ is a set composed of a finite number of eigenvalues and the sum of their multiplicity is $k$, then this is also true for $\sigma(A+\epsilon B_\epsilon)\cap [a_1,a_2]$ for all $|\epsilon|<\epsilon_A$.
\end{enumerate}
\end{theorem}
This theorem corresponds to \cite[Theorem~5.6]{L22} when $B_\epsilon$ does not depend on $\epsilon$. It is also related to \cite[Chapter~7, Theorem~1.8]{Kato}. Here, we follow
the proof of \cite[Theorem~5.6]{L22} and carefully examine that Theorem~\ref{th-Le} can be proved in the same way.
\begin{enumerate}
 \item The assertion (a) comes from a general theorem which is independent of $\epsilon$, namely applying \cite[Theorem~5.2]{L22} because $\| \epsilon B_\epsilon (A-a)^{-1}\|<1$.
 \item We consider $\cal C$ a Jordan curve in $\C$, surrounding $[a_1,a_2]$ and crossing the real axis only in $a_1$ and $a_2$ (for instance, a rectangle $\partial([a_1,a_2]\times [-M,M])$) and we prove the following estimate by using the definition of $\epsilon_A$
 \[
 \Big\| \epsilon B_\epsilon (A-z)^{-1}\Big\| \leq \Big(1+\frac{a_2-a_1}{2M}\Big) \frac{|\epsilon|}{\epsilon_A} \quad \forall z\in \cal C.
 \]
 For every $|\epsilon|<\epsilon_A$, choosing $M$ large enough allows to write $\epsilon \mapsto (A+\epsilon B_\epsilon-z)^{-1}$ as a convergent series with respect to $\epsilon$:
 \[
 (A+\epsilon B_\epsilon-z)^{-1} = (A-z)^{-1} \sum_{n\geq 0} (-1)^n \epsilon^n (B_\epsilon (A-z)^{-1})^n
 \]
but as $\epsilon \mapsto B_\epsilon$ is analytic, we deduce that $\epsilon\mapsto (A+\epsilon B_\epsilon-z)^{-1}$ is analytic, which is exactly what is needed to finish the proof in Lewin's Lectures Notes.
\item We follow the proof in \cite[Theorem~5.6]{L22} by establishing the analyticity of the spectral projector $P(\epsilon)$ and by using that the rank of an orthogonal projector is an entire and continuous function, then remaining constant.
\end{enumerate}
We will apply the above theorem with 
\[
A=G_\theta[0], \quad \epsilon B_\epsilon= G _\theta[\epsilon b]- G_\theta[0], \quad D(A)=H^1(\T_{2\pi}), \quad H=L^2(\T_{2\pi}).
\]
From the analycity of the resolvent (see Proposition~\ref{proposition_analyticity_in_epsilon}), we write
\[
1 + G_\theta[\epsilon b] = \Big(R_0( 1 + \sum_{k\geq 1} \epsilon^k R_0^{-1} R_k )\Big)^{-1}=( 1 + \sum_{k\geq 1} \epsilon^k R_0^{-1} R_k )^{-1} R_0^{-1}
\]
where we notice that $R_0^{-1}R_k$ is a bounded operator from $L^2(\T_{2\pi})$ to $L^2(\T_{2\pi})$, with a norm less than $C_0r^{-k}$. Recalling that $R_0^{-1} = 1+ G_\theta[0]$, this allows us to identify $B_\epsilon$ as an analytic function in $\cal L(H^1(\T_{2\pi});L^2(\T_{2\pi}))$ and state that the boundedness in $\cal L(L^2(\T_{2\pi}))$ of $R_0^{-1}(A+i)^{-1}$ gives that $B_\epsilon(A+i)^{-1}$ is uniformly bounded. 

The first proposition shows that, for $\theta$ not too close to $0$ or $\frac{1}{2}$, $p\in \N$ and $\epsilon$ sufficiently small (depending on $b$ and $p$), $ G_\theta[\epsilon b]$ has a simple eigenvalue $\lambda_p^{\epsilon} (\theta)$ in an interval outside the gap we will construct.
Recall that $F_p$ is defined in \eqref{def-Fp}.

\begin{proposition}[Perturbation of a simple eigenvalue] \label{simple-ev}
Fix $p \in \mathbb{N}$. There exist $\epsilon_{p,1}>0$, $d_{p,1}>0$ and $d_{p,2}\geq F_{2p}|\widehat{b}_{2p}|+F_{2p+1}|\widehat{b}_{2p+1}|+F_{2p+2}|\widehat{b}_{2p+2}|$ depending on $p$ and $b$, such that, for all $\epsilon \in [0,\epsilon_{p,1})$, and $d_{p,1}\epsilon \le \theta \le \frac{1}{2} -d_{p,1}\epsilon$, we have
\begin{equation}\label{eq:sigmasimple}
 \sigma(G_\theta[\epsilon b]) \cap \Big[\lambda_{2p}^{0} (0) , \lambda_{2p+2}^{0} (0)\Big] = \{\lambda_{2p}^{\epsilon} (\theta), \lambda_{2p+1}^{\epsilon} (\theta) \}
\end{equation}
where $\lambda_{2p}^{\epsilon} (\theta)$ and $\lambda_{2p+1}^{\epsilon} (\theta)$ are simple. Moreover,
\begin{align*}
 0&\leq \lambda_{2p}^{\epsilon} (\theta) \leq\lambda_{2p}^{0} (\tfrac12) - d_{p,2}\epsilon, \quad \text{if }p=0,\\
 \lambda_{2p}^{0} (0) + d_{p,2}\epsilon&\leq \lambda_{2p}^{\epsilon} (\theta) \leq\lambda_{2p}^{0} (\tfrac12) - d_{p,2}\epsilon, \quad \text{if }p>0,\\
 \lambda_{2p}^{0} (\tfrac12) + d_{p,2}\epsilon&\leq \lambda_{2p+1}^{\epsilon} (\theta) \leq\lambda_{2p+2}^{0} (0) - d_{p,2}\epsilon,
\end{align*}
and 
\[
\max_{k=2p,2p+1} \max_{d_{p,1}\epsilon \leq \theta \leq \frac{1}{2} -d_{p,1}\epsilon} \Big| \lambda_{k}^{\epsilon}(\theta) -\lambda_{k}^{0}(\theta) \Big| \leq d_{p,2} \epsilon.
\]
\end{proposition}

The condition $d_{p,2}\geq F_{2p}|\widehat{b}_{2p}|+F_{2p+1}|\widehat{b}_{2p+1}|+F_{2p+2}|\widehat{b}_{2p+2}|$ does not appear naturally in the proof, but it will be necessary for the opening of the gap proved in Section~\ref{section:4}.

\begin{proof}
Fix $p \in \mathbb{N}$. As it was noted before Proposition~\ref{simple-ev} that $B_\epsilon(A+i)^{-1}$ is uniformly bounded, we have also that $\|\epsilon B_\epsilon(A+1)^{-1} \| <1$ for all $\epsilon\in [0,\epsilon_0(b))$. We set
\begin{equation}\label{def-Reta}
R_p:=\lambda^0_{2p+2}(0)+1 , \quad \eta_\epsilon:= 1 -\|\epsilon B_\epsilon(A+1)^{-1} \| 
\end{equation}
which verifies $1\leq R_p<\eta_\epsilon/\|\epsilon B_\epsilon(A+1)^{-1} \| $ for $\epsilon$ small enough (depending only on $b$ and $p$). We choose $d_{p,2}$ larger than $F_{2p}|\widehat{b}_{2p}|+F_{2p+1}|\widehat{b}_{2p+1}|+F_{2p+2}|\widehat{b}_{2p+2}|$ such that
\[
\frac{R^2_p\|\epsilon B_\epsilon(A+1)^{-1} \|}{\eta_\epsilon -R_p\|\epsilon B_\epsilon(A+1)^{-1} \|} \leq d_{p,2}\epsilon, \quad \forall \epsilon\in [0,\epsilon_{p,1}], \ \theta \in [0,\frac{1}{2}]
\]
where $\epsilon_{p,1}>0$ is chosen small enough, depending only on $b$ and $p$.
As $\frac{\mathrm{d}\lambda_{k}^{0}}{\mathrm{d}\theta}(0)\neq 0$ if $k>0$ and $\frac{\mathrm{d}\lambda_{k}^{0}}{\mathrm{d}\theta}(\frac12) \neq 0$ if $k\geq 0$, we can fix $d_{p,1}>0$ such that\footnote{For instance, for $k>0$, $\frac{\mathrm{d}\lambda_{2k}^{0}}{\mathrm{d}\theta}(0),\frac{\mathrm{d}\lambda_{2k}^{0}}{\mathrm{d}\theta}(\frac12) \neq 0$ implies that $\lambda_{2k}^{0}(0)+c_k \theta \leq\lambda_{2k}^{0}(\theta) \leq \lambda_{2k}^{0}(\tfrac12) - c_k (\tfrac12-\theta)$ for all $\theta\in [0,1/2]$.}
\begin{equation}\label{sys3.13}
 \begin{aligned}
 \lambda_{0}^{0} (\theta) < \lambda_{0}^{0} (\tfrac12) -2d_{p,2}\epsilon , \quad &\text{for } \theta\in [0,\tfrac12-d_{p,1}\epsilon]\\
 \lambda_{2k}^{0} (0) + 2d_{p,2}\epsilon < \lambda_{2k}^{0} (\theta) < \lambda_{2k}^{0} (\tfrac12) -2d_{p,2}\epsilon , \quad &\text{for }0< k \leq p \\
 &\text{and }\theta\in [d_{p,1}\epsilon,\tfrac12-d_{p,1}\epsilon],\\
 \lambda_{2k}^{0} (\tfrac12) + 2d_{p,2}\epsilon < \lambda_{2k+1}^{0} (\theta) < \lambda_{2k+2}^{0} (0) -2d_{p,2}\epsilon , \quad &\text{for }0\leq k \leq p \\
 &\text{and }\theta\in [d_{p,1}\epsilon,\tfrac12-d_{p,1}\epsilon],\\
\end{aligned}\end{equation}

We assume now that $p>0$ and we will comment later the case $p=0$. 

Let $a_1 = \lambda_{2p}^{0} (0)$ and $a_2 = \lambda_{2p}^{0}(\tfrac{1}{2})$ which clearly belong to the resolvent set of $A$ when $\theta\in (0,1/2)$. For $\theta \in [d_{p,1}\epsilon,\tfrac12-d_{p,1}\epsilon]$, the previous inequalities imply that 
\begin{equation*}
\mathrm{d}(a_1, \sigma(G_\theta[0])) \geq 2d_{p,2}\epsilon, \quad \mathrm{d}(a_2, \sigma(G_\theta[0])) \geq 2d_{p,2}\epsilon,
\end{equation*}
hence by \eqref{est-BA}
\begin{equation*}
 \| B_\epsilon (G_\theta[0]-a_1)^{-1}\| \le \frac{C_{b,p} }{2d_{p,2}\epsilon} , \quad 
\| B_\epsilon (G_\theta[0]-a_2)^{-1}\| \le \frac{C_{b,p}}{2d_{p,2}\epsilon} .
\end{equation*} 
 As for Theorem~\ref{th-Le}, we set 
 \begin{equation*}
 \epsilon_A = \inf\Big( \frac{1}{2\|B_\epsilon(A-a_1)^{-1} \|}; \frac{1}{2\|B_\epsilon(A-a_2)^{-1}\|}; \epsilon_0 \Big)
 \geq \min\Big( \frac{d_{p,2}\epsilon}{C_{b,p}}; \epsilon_0 \Big) .
\end{equation*}
Without any loss of generality, we can assume that $d_{p,2}$ was chosen large enough such that $d_{p,2}\geq C_{b,p}$. Therefore, for any $\epsilon \in [0,\epsilon_{p,1})$ and $d_{p,1}\epsilon \le \theta \le \frac12-d_{p,1}\epsilon$, we have $0\leq \epsilon < \epsilon_A $, so Theorem~\ref{th-Le} implies that there exists a unique eigenvalue inside $[\lambda_{2p}^{0} (0), \lambda_{2p}^{0} (\frac12)]$. Moreover, it is simple and strictly included in this interval.
 Applying the same argument with $a_1=\lambda_{2p}^{0} (\frac12)$ and $a_2= \lambda_{2p+2}^{0} (0)$, there exists a unique eigenvalue inside $[\lambda_{2p}^{0} (0), \lambda_{2p}^{0} (\frac12)]$. Moreover, it is simple and strictly included in this interval.

For $p=0$, we simply consider $a_1=-1$ and $a_2=\lambda_{2p}^{0} (\frac12)$, which means that for $\theta \in [0,\frac12-d_{p,1}\epsilon]$, there is a unique eigenvalue inside $[-1, \lambda_{0}^{0} (\frac12)]$. Moreover, it is simple, strictly included in this interval and is non negative by the positivity of $G_\theta[\epsilon b]$. Therefore, choosing $p\mapsto \epsilon_{p,1}$ decreasing, we can count the eigenvalues and conclude the proof of \eqref{eq:sigmasimple}.

Next, we apply \cite[Theorem 5.2]{L22} with $a=-1$, $R_p$ and $\eta_\epsilon$ defined in \eqref{def-Reta}, to state that 
\begin{align*}
\sup_{k \leq 2p+2} \inf_{j \in \mathbb{N}} |\lambda_j^\epsilon(\theta) - \lambda_k^0(\theta)| &\leq 
\frac{R_p^2\|\epsilon B_\epsilon(A+1)^{-1} \|}{\eta_\epsilon -R_p\|\epsilon B_\epsilon(A+1)^{-1} \|} \leq d_{p,2}\epsilon , \\ 
\sup_{j \leq 2p+2} \inf_{k \in \mathbb{N}} |\lambda_j^\epsilon(\theta) - \lambda_k^0(\theta)| &\leq \frac{R_p^2\|\epsilon B_\epsilon(A+1)^{-1} \|}{\eta_\epsilon -R_p\|\epsilon B_\epsilon(A+1)^{-1} \|} \leq d_{p,2}\epsilon .
\end{align*}
Using \eqref{sys3.13}, we conclude that the infimum is reached for $j=k$, which ends the proof of the proposition.
\end{proof}

\begin{remark}\label{rem-simple0}
 If $p=0$, we note in the previous proof that we have the information for $\lambda^\epsilon_0(\theta)$ up to $\theta =0$, namely for all $\epsilon\in [0,\epsilon_{p,1})$ and $\theta\in [0,\frac12-d_{p,1}\epsilon]$, one has
 \[
 \sigma(G_\theta[\epsilon b]) \cap \Big[0, \lambda^0_{0}(\tfrac12)\Big] = \{\lambda_0^\epsilon(\theta) \},
 \]
 where $\lambda_0^\epsilon(\theta)$ is simple, belongs to $[0,\lambda^0_{0}(\tfrac12)-d_{p,2}\epsilon]$ and is such that 
 \[\max_{0 \leq \theta \leq \frac{1}{2} -d_{p,1}\epsilon} \Big| \lambda_{0}^{\epsilon}(\theta) -\lambda_{0}^{0}(\theta) \Big| \leq d_{p,2} \epsilon.\]
\end{remark}

The next proposition provides a first description of the spectrum of $G_\theta[\epsilon b]$ for $\epsilon $ small, and $\theta$ close to $0$ or $\frac{1}{2}$, where $G_\theta [0]$ has a eigenvalue of multiplicity two. The next section will give conditions on the bottom $b$ that lead to the separation of the double eigenvalue into two simple eigenvalues, creating a gap.

\begin{proposition}[Perturbation of a double eigenvalue] \label{double-ev0}
Fix $p \in \mathbb{N}^\ast$. There exist $\epsilon_{p,2}\in (0,\epsilon_{p,1}]$ and $d_{p,3}>0$ depending on $p$ and $b$ such that, for all $\epsilon \in [0,\epsilon_{p,2}]$, and $0 \le \theta \le d_{p,1}\epsilon$, we have that
\begin{equation*}
 \sigma(G_\theta[\epsilon b]) \cap \Big[\lambda_{2p-1}^{0} (\tfrac12) , \lambda_{2p}^{0} (\tfrac12)\Big] 
\end{equation*}
contains exactly two eigenvalues counted with multiplicity $ \lambda_{2p-1}^{\epsilon}(\theta) \leq \lambda_{2p}^{\epsilon}(\theta)$.
 Moreover,
\[
\max_{k=2p-1,2p} \max_{0 \leq \theta \leq d_{p,1}\epsilon} \Big| \lambda_{k}^{\epsilon}(\theta) -\lambda_{k}^{0}(\theta) \Big| \leq d_{p,3} \epsilon.
\]
\end{proposition}

\begin{proof}
The proof follows the strategy of the proof of Proposition~\ref{simple-ev} but is much simpler as $\lambda^0_{2p}(\theta)$ is far from $\lambda^0_{2p}(\frac12)$ for $\theta \in [0, d_{p,1}\epsilon]$.
Fix $p \in \mathbb{N}^*$, then
\begin{equation*}
 \lambda_{2p-1}^{0} (\tfrac12) + \tfrac18 < \lambda_{2p-1}^{0} (\theta) \leq \lambda_{2p}^{0} (\theta) < \lambda_{2p}^{0} (\tfrac12) -\tfrac18 , \quad \text{for } \theta\in [0,d_{p,1}\epsilon]\\
\end{equation*}
for all $\epsilon$ small enough, depending on $p$ and $d_{p,1}$.

Let $a_1 = \lambda_{2p-1}^{0} (\tfrac12)$ and $a_2 = \lambda_{2p}^{0}(\tfrac{1}{2})$ which clearly belong to the resolvent set of $A$ when $\theta\in [0,d_{p,1}\epsilon ]$. As
$\mathrm{d}(a_1, \sigma(G_\theta[0])),\mathrm{d}(a_2, \sigma(G_\theta[0])) \geq 1/4$, \eqref{est-BA} gives
\begin{equation*}
 \| B_\epsilon (G_\theta[0]-a_1)^{-1}\| \le C_{b,p} , \quad 
\| B_\epsilon (G_\theta[0]-a_2)^{-1}\| \le C_{b,p} .
\end{equation*} 
 As for Theorem~\ref{th-Le}, we set 
 \begin{equation*}
 \epsilon_A = \inf\Big( \frac{1}{2\|B_\epsilon(A-a_1)^{-1} \|}; \frac{1}{2\|B_\epsilon(A-a_2)^{-1}\|}; \epsilon_0 \Big)
 \geq \min\Big( \frac1{2C_{b,p}}; \epsilon_0 \Big) .
\end{equation*}
Therefore, for any $\epsilon \in [0,\epsilon_{p,2})$, where $\epsilon_{p,2}$ is small enough, and $0 \le \theta \le d_{p,1}\epsilon$, we have $0\leq \epsilon < \epsilon_A $, so Theorem~\ref{th-Le} implies that there exist two eigenvalues counted with multiplicity inside $[\lambda_{2p-1}^{0} (\tfrac12), \lambda_{2p}^{0} (\frac12)]$. Moreover, they are strictly included in this interval.

Choosing $p\mapsto \epsilon_{p,2}$ decreasing, we can count the eigenvalues and conclude that they correspond to $\lambda_{2p-1}^\epsilon(\theta)$ and $\lambda_{2p}^\epsilon(\theta)$.

We next use again that for $\epsilon \in [0,\epsilon_{p,1}]$, we have $\|\epsilon B_\epsilon(A+1)^{-1} \| <1$ and $1\leq R_p<\eta_\epsilon/\|\epsilon B_\epsilon(A+1)^{-1} \| $ where
\[
R_p =\lambda^0_{2p+2}(0)+1 , \quad \eta_\epsilon:= 1 -\|\epsilon B_\epsilon(A+1)^{-1} \| 
\]
So \cite[Theorem 5.2]{L22} with $a=-1$ gives that 
\begin{align*}
\sup_{k \leq 2p+2} \inf_{j \in \mathbb{N}} |\lambda_j^\epsilon(\theta) - \lambda_k^0(\theta)| &\leq 
\frac{R_p^2\|\epsilon B_\epsilon(A+1)^{-1} \|}{\eta_\epsilon -R_p\|\epsilon B_\epsilon(A+1)^{-1} \|} \leq d_{p,2}\epsilon \\ 
\sup_{j \leq 2p+2} \inf_{k \in \mathbb{N}} |\lambda_j^\epsilon(\theta) - \lambda_k^0(\theta)| &\leq \frac{R_p^2\|\epsilon B_\epsilon(A+1)^{-1} \|}{\eta_\epsilon -R_p\|\epsilon B_\epsilon(A+1)^{-1} \|} \leq d_{p,2}\epsilon .
\end{align*}
The first inequalities established in this proof end the proof of the proposition because $|\lambda_{2p}^0(\theta) - \lambda_{2p-1}^0(\theta)|\leq C_p d_{p,1}\epsilon$ for all $\theta\in [0,d_{p,1}\epsilon]$.
\end{proof}

The previous proof can be directly adapted in the neighborhood of $1/2$.
\begin{proposition}[Perturbation of a double eigenvalue] \label{double-ev1-2}
Fix $p \in \mathbb{N}$. There exist $\epsilon_{p,3}\in (0,\epsilon_{p,1}]$ and $d_{p,4}>0$ depending on $p$ and $b$ such that, for all $\epsilon \in [0,\epsilon_{p,3}]$, and $\frac12 -d_{p,1}\epsilon \le \theta \le \frac12$, we have that
\begin{equation*}
 \sigma(G_\theta[\epsilon b]) \cap \Big[\lambda_{2p}^{0} (0) , \lambda_{2p+2}^{0} (0)\Big] 
\end{equation*}
contains exactly two eigenvalues counted with multiplicity $ \lambda_{2p}^{\epsilon}(\theta) \leq \lambda_{2p+1}^{\epsilon}(\theta)$.
 Moreover,
\[
\max_{k=2p,2p+1} \max_{\frac12-d_{p,1}\epsilon \leq \theta \leq \frac12 } \Big| \lambda_{k}^{\epsilon}(\theta) -\lambda_{k}^{0}(\theta) \Big| \leq d_{p,4} \epsilon.
\]
\end{proposition}

\subsection{Proof of Theorem~\ref{union-of-bands}}

 To see more clearly the role of parameters $\theta$ and $\epsilon$, we use in this section the notation $\lambda_p(\theta,\epsilon)$ for $\lambda_p^\epsilon(\theta)$, and $\tau_p(\theta,\epsilon)$ for $\tau_p^\epsilon(\theta)$.

Theorem~\ref{union-of-bands} will follow from general perturbation theory for analytic operators. 
We recall that the compactness and self-adjointness of the resolvent $(1+G_\theta[\epsilon b])^{-1}$ provide eigenvalues $(\lambda_n(\theta,\epsilon))_n $ for $G_\theta[\epsilon b]$ (see Remark~\ref{rem-eigen-resolvent}) and the associated eigenfunctions $(\psi_n(\theta,\epsilon))_n$ form a complete orthonormal basis of $L^2(\T_{2\pi})$.

For any fixed $\epsilon$, the analyticity of the resolvent with respect to $\theta$ (Proposition~\ref{proposition_analyticity_in_theta}) allows us to apply in \cite[Chapter II, Theorem~1]{Rellich} or \cite[Chapter~VII, Theorem~3.9]{Kato} to state that there is a reordering $(\tilde\lambda_n,\tilde\psi_n)$ of $(\lambda_n,\psi_n)$ such that the functions $\theta \mapsto \tilde\lambda_n(\theta, \epsilon)$ and $\theta \mapsto \tilde\psi_n(\theta, \epsilon)$ are analytic in a neighborhood of $[-\frac12,\frac12]$ (a necessary reordering near crossing eigenvalues, see Figure~\ref{fig}).

As illustrated in Figure~\ref{fig}, having analytic eigenvalues with respect to $\theta$ in the neighborhood of a crossing point means that $(\tilde\lambda_{n})_{n}$ is not necessary an increasing sequence for all $\theta$. Alternatively, we can redefine the functions $\theta \mapsto (\lambda_n,\psi_{n})(\theta, \epsilon)$ so that the eigenvalues are in increasing order, but, in this case, we can only say that the eigenvalues are Lipschitz with respect to $\theta \in (-\frac12,\frac12]$. This is the choice made in \cite[Theorem~7.3]{L22}.

We now prove the relation given in Theorem~\ref{union-of-bands} between the spectrum of $G[\varepsilon b]$ and $\lambda_{n}$, using general argument of the Bloch-Floquet theory. A way to verify this statement is to extend the proof of such an equality for the Schr\"odinger operator \cite[Theorem~7.3]{L22}. 
We start with the inclusion from right to left.
Define $f_{\theta}:=\psi_{n}(\theta, \varepsilon)\in H^1(\T_{2\pi})$ the eigenfunction associated to the eigenvalue $\lambda_{n}(\theta,\varepsilon)$ of $G_{\theta}[\varepsilon b]$ and set
\[
g_{\eta}(x):=\eta^{1/2} e^{i\theta x} f_{\theta}(x) \chi (\eta x)\in H^1(\R), 
\]
where $\chi \in C^\infty_{c}(\R,\R_+)$ and $\int \chi^2=1$. It is proved therein that $\|g_{\eta}\|_{L^2}\to (2\pi)^{-1/2}$ as $\eta\to 0$. The only point to adapt is the fact that $(G[\varepsilon b]-\lambda_{n})g_{\eta}$ tends to zero in the limit $\eta\to 0$. For this, let us consider $\Psi_{\eta}(x,z)= \eta^{1/2} e^{i\theta x} \Phi_{\theta}(x,z) \chi (\eta x)$ where $\Phi_{\theta}$ is the solution of the elliptic problem \eqref{D2N-theta} associated to $G_{\theta}[\varepsilon b]\psi_{n}$
and verify that it satisfies in the sense of distributions
\begin{equation*}
 \begin{cases}
 -\Delta \Psi_{\eta} = - \eta^{5/2} e^{i\theta x} \Phi_{\theta} \chi'' (\eta x) - 2\eta^{3/2} e^{i\theta x}(i\theta \Phi_{\theta}+ \partial_{x} \Phi_{\theta}) \chi' (\eta x) \text{ in }\Omega_\epsilon, \\
 \Psi_{\eta}\vert_{z=0} = g_{\eta}, \quad \partial_n \Psi_{\eta} \vert_{z=-1+\epsilon b} = \eta^{3/2} n_{x} e^{i\theta x} \Phi_{\theta} \chi' (\eta x) , \quad \partial_n \Psi_{\eta} \vert_{z=0} = \lambda_{n}g_{\eta}
\end{cases}
 \end{equation*}
whereas the solution $\Phi$ of the elliptic problem \eqref{D2N-theta} associated to $G[\varepsilon b] g_{\eta}$ satisfies
\[
 \begin{cases}
 -\Delta \Phi = 0 \quad \text{in }\Omega_\epsilon, \\
 \Phi_{\vert_{z=0}} = g_{\eta}, \quad \partial_n \Phi _{\vert_{z=-1+\epsilon b}} = 0.
\end{cases}
\]
This implies that
\begin{align*}
\|(G[\varepsilon b]&-\lambda_{n})g_{\eta} \|_{L^2(\R)}= \| (\partial_{n}\Phi - \partial_{n}\Psi_{\eta} ) \vert_{z=0} \|_{L^2(\R)} \leq C \| \Phi -\Psi_{\eta} \|_{H^{3/2}(\Omega_{\varepsilon})}\\
 \leq &C \| - \eta^{5/2} e^{i\theta x} \Phi_{\theta} \chi'' (\eta x) - 2\eta^{3/2} e^{i\theta x}(i\theta \Phi_{\theta}+ \partial_{x} \Phi_{\theta}) \chi' (\eta x) \|_{H^{1/2}(\Omega_{\varepsilon})} \\
 &+C \| \eta^{3/2} n_{x} e^{i\theta x} \Phi_{\theta} \chi' (\eta x) \|_{L^2(\partial \Omega_{\varepsilon})}.
\end{align*}
The right hand side term tends to zero as in \cite[Theorem~7.3]{L22}, which implies, from the Weyl's criterion (see \cite[Theorem~VII.12]{RS80}), that $\lambda_n\in \sigma (G[\varepsilon b])$.

For the converse, we notice that Theorem~\ref{theo:Floquet-decompop} together with the isometry of $\cal U$ gives
\[
\| G[\varepsilon b] f \|_{L^2(\R)}^2 = \int_{0}^1 \int_{0}^{2\pi} \Big|G_{\theta}[\varepsilon b] \cal U f(\cdot, \theta)\Big|^2(x) \dd x \dd \theta
 \]
so the rest of the proof of \cite[Theorem~7.3]{L22} can be readily applied. 

Notice that $G_0[\epsilon b] 1 =0$ so $\lambda_0^\epsilon(0)=0$. 

Part~(ii) of Theorem~\ref{union-of-bands} is a direct consequence of Proposition~\ref{simple-ev}, Remark~\ref{rem-simple0} and Propositions~\ref{double-ev0} and \ref{double-ev1-2}, by setting $\epsilon_1(b,p) =\min (\epsilon_{1,p},\epsilon_{2,p},\epsilon_{3,p})$.

We now prove (iii). We first fix $N_1$ such that for all $\theta \in [-\frac{1}{2},\frac{1}{2}]$, $\lambda_{N_1+1}^0(\theta) \geq M+2$. By Part~(ii) of the theorem, there exists $\epsilon_1(b,N_1)$ such that for all $\epsilon \leq \epsilon_1(b,N_1)$ and for all $\theta \in [-\frac{1}{2},\frac{1}{2}]$, $\lambda_{N_1+1}^\epsilon(\theta) \geq M+1$. Therefore for all $n \geq N_1+1$,
\begin{equation}\label{lambda_n_larger_M}
\lambda_n^\epsilon(\theta) \geq M+1.
\end{equation}
Furthermore, there also exists $N$ such that for all $n \geq N+1$, 
\begin{equation}\label{lambda_tilde_n_larger_M}
\tilde \lambda_n^\epsilon(\theta) \geq M+\tfrac12.
\end{equation}
We need the above estimate on $\tilde \lambda_n^\epsilon(\theta)$ in order to reproduce the proof Theorem XIII.86 of \cite{RS78} and obtain the purely absolutely continuous character of the lower part of the spectrum of $\sigma(G[\epsilon b])$.
To prove \eqref{lambda_tilde_n_larger_M}, we proceed by contradiction.
Assume that for all $N$, there exist $n \geq N+1$ and $\theta_n \in [-\frac12,\frac12]$ such that $\tilde \lambda_{n}^\epsilon(\theta_n) < M+\tfrac12$, i.e. there is an infinite number of $\tilde \lambda_{n}^\epsilon$ passing from $\tilde \lambda_{n}^\epsilon (0)=\lambda_{n}^\epsilon (0)>M+1$, to values less that $M+ \tfrac12$. There is no subsequence such that $(\theta_n)$ is constant (say equal to $\theta_0)$, because this would imply that there is an infinite number of eigenvalues at $\theta_0$ less than $M+1$. Therefore, we can extract a subsequence $(\theta_n)$ which is increasing or decreasing, say increasing, where $\tilde \lambda_n(\theta_n)< M+\frac12$. Denote $\theta_\infty \in [-\frac12,\frac12]$ the limit of $(\theta_n)$. Since there is only a finite number of eigenvalues $(\tilde \lambda_n^\epsilon(\theta_\infty))$ less than $M+1$, an infinite number of $\tilde \lambda_n^\epsilon$ has to exit $[0,M+\frac34]$, i.e. there are $c_n\in [\theta_n,\theta_\infty)$ such that $M+1 \geq \tilde \lambda_n^\epsilon(c_n)\geq M+\frac34$. As $\tilde \lambda_n^\epsilon$ can be defined only from $N_1+1$ functions $\lambda_{n}^\epsilon$, there is an index $\varphi(n)\in [0,N_1]$ and some values $a_n, b_n$ of the parameter $\theta$ such that, $\theta_n\leq a_n\leq b_n\leq c_n$ and $|\lambda_{\varphi(n)}^\epsilon(b_n) -\lambda_{\varphi(n)}^\epsilon(a_n)|\geq \frac1{4(N_1+1)}$. As $\varphi(n)\in [0,N_1]$, we can extract a sequence such that $\varphi(n)$ is constant (say $N_\infty$), which is impossible because $\lambda_{N_\infty}^\epsilon$ is Lipschitz.
Therefore \eqref{lambda_tilde_n_larger_M} holds.

We define 
\begin{equation}
 H_n := \{h \in L^2([-\tfrac{1}{2},\tfrac{1}{2}],\mathbb{T}) \; ; \; h(\theta) = f(\theta)\tilde\psi_n^\epsilon(\theta, \epsilon), f \in L^2([-\tfrac{1}{2},\tfrac{1}{2}]) \}
\end{equation}
and $G_n : H_n \longrightarrow H_n$ as the multiplication by the function $\tilde\lambda_n(\cdot, \epsilon)$.

We use Propositions~\ref{simple-ev}, \ref{double-ev0} and \ref{double-ev1-2} to state that there exists $\epsilon_2(N,b)$ such that for all $n \leq N$ and for all $\epsilon \leq \epsilon_2(N,b)$ and $\theta \in [-\frac{1}{2},\frac{1}{2}]$, $|\tilde\lambda_n^\epsilon(\theta) - \tilde\lambda_n^0(\theta)| \leq C_{N,b}\epsilon$. Since $\tilde \lambda_n^0(\theta) = (r+\theta)\tanh(r+\theta)$ for some $r \in \mathbb{N}$, we get that for $\epsilon$ small enough and for all $n \leq N$, $\tilde\lambda_n^\epsilon(\cdot)$ is not constant. The proof of \cite[Theorem XIII.86]{RS78} then ensures that $G_n$ has only absolutely continuous spectrum. Note also that $\sigma(G_n) = \tilde \lambda_n^\epsilon([-\frac{1}{2},\frac{1}{2}])$ so by Part~(i) of the theorem and \eqref{lambda_tilde_n_larger_M} we have
\begin{equation*}
 \sigma(G[\epsilon b])\cap[0,M] = \bigcup_{n=0}^N\sigma(G_n)\cap[0,M]
\end{equation*}
and therefore (iii).

\section{Gap opening of order $\epsilon$}\label{section:4}

Theorem~\ref{union-of-bands} shows that the spectrum of $G[\epsilon b] $ is composed of union of bands that may or may not overlap. In this section, we show that for a given $p$, if $\widehat{b}_{2p}\neq 0$, a gap of size $\epsilon$ occurs between $\lambda^\epsilon_{2p-1}$ and $\lambda^\epsilon_{2p}$, namely
\begin{equation*}
 \max_{-\frac{1}{2} < \theta\leq\frac{1}{2}} \lambda_{2p-1}^\epsilon(\theta) =\max_{0\le\theta \leq\frac{1}{2}} \lambda_{2p-1}^\epsilon(\theta)<
 \min_{-\frac{1}{2}<\theta\leq \frac{1}{2}}
 \lambda_{2p}^\epsilon(\theta) =\min_{0\le\theta\leq \frac{1}{2}}
 \lambda_{2p}^\epsilon(\theta) ,
\end{equation*}
where we have used the evenness of the spectrum. 

In the case of one-dimensional Schr\"odinger operators with periodic potentials, bands cannot overlap due to the key property that the eigenvalues, labeled in increasing order are strictly monotone functions of $\theta$, and studying the opening of a gap then reduces to studying the splitting of the eigenvalues at $\theta=0,1/2$. 
However, for $G_\theta[\epsilon b]$, the monotonicity of the $\lambda_p(\theta, \epsilon)$ with respect of $\theta$ is unknown. 
The opening of gaps happens in the neighborhood of $\theta=0,1/2$, and a detailed matching of the inner and outer regions must be done, to ensure that the gap indeed exists.

The main idea of this section and Section~\ref{section:5} is the construction of approximated eigenvalues, and we will use the approximation lemma below which is an extension of a result of Bambusi, Kappeler and Paul, \cite[Proposition~5.1]{BKP15} for operators on finite dimensional spaces, to compact operators on Hilbert spaces. 

\begin{lemma}\label{lemma:quasimodes}
Let $K$ be a compact positive semi-definite self-adjoint operator on a separable Hilbert space $H$.

If $(\lambda^{\app},u^{\app}) \in \mathbb{R}_+ \times H$ satisfies $\norm{u^{\app}} = 1$ and $\norm{K u^{\app} - \lambda^{\app}u^{\app}} \leq \cal E$, then there exists an eigenvalue $\lambda$ of $K$ such that $|\lambda - \lambda^{\app}| \leq \cal E$. 
\end{lemma}

\begin{proof}
Let $\{\lambda_n\}_{n\in \N}$ be the set of eigenvalues of $K$ (counted with multiplicity), associated to an orthonormal basis of unitary eigenvectors $\{u_n\}_{n\in \N}$. 
\begin{equation*}
K u^{\app} - \lambda^{\app}u^{\app} = \sum_{n} \lambda_n(u^{\app},u_n)u_n - \sum_n\lambda^{\app}(u^{\app},u_n)u_n . 
\end{equation*}
Therefore,
\begin{equation*}
\norm{K u^{\app} - \lambda^{\app}u^{\app}}^2 = \sum_{n} (\lambda_n - \lambda^{\app})^2(u^{\app},u_n)^2
\geq \inf_{n \in \mathbb N}|\lambda_n - \lambda^{\app}|^2.
\end{equation*}
Thus, $\inf_{n \in \mathbb N}|\lambda_n - \lambda^{\app}| \leq \cal E$. Let $(\lambda_{n_k})\subset \R_+$ be a subsequence such that 
\begin{equation*}
|\lambda_{n_k} - \lambda^{\app}| \Tend{k}{+\infty} \inf_{n \in \mathbb N}|\lambda_n - \lambda^{\app}|.
\end{equation*}
Since there is no accumulation of eigenvalues outside zero and as $\lambda^{\app}+\cal E>0$, there exists $k$ such that $\lambda_{n_k} \in [\lambda^{\app} - \cal E, \lambda^{\app} + \cal E ]$.
\end{proof}

\subsection{Perturbation of double eigenvalues}

Fix $p \in \mathbb{N}^*$. We will prove that, under the assumptions of Theorem~\ref{one_gap-order-eps} part (i) and for $\theta$ small enough, the spectrum of $G_\theta[\epsilon b]$ near $\lambda_{2p}^0(0)$ is composed of two eigenvalues $\lambda_{p\pm}$ separated by a gap of size $\epsilon$. We will use asymptotic expansions to create two approximate eigenvalues $\lambda_{p\pm}^{\app}$ separated by a gap of size $\epsilon$ and show that $\lambda_{p+}$ (resp $\lambda_{p-}$) is in an $\epsilon^{2}$ neighborhood of $\lambda^{\app}_{p+}$ (resp $\lambda^{\app}_{p-}$). This argument relies on Lemma~\ref{lemma:quasimodes}.

To construct an approximate solution of the system \eqref{equation_eigenvalue_problem_S1} for $\epsilon$ small and $\theta \in [0,d_{1,p}\epsilon]$, we use an idea of Chiad\`o-Piat et al \cite[Section 3]{CNR13} and consider simultaneously the two small parameters $\epsilon$ and $\theta$. 

Fix $\theta = \delta\epsilon$ with $\delta \in [0,d_{1,p}]$ ($d_{1,p}>0$ is given in Proposition~\ref{simple-ev} and may be large), and write the following Ansatz for the approximate eigenpair $(\lambda^{\app}_{p\pm}, U^{\app}_{p\pm}) $ in the neighborhood of $\theta = 0$: 
\begin{equation} \label{Ansatz-N}
\left\{
\begin{aligned}
\lambda_{p\pm}(\theta = \delta\epsilon,\epsilon) &\approx \lambda_{2p}^0(0) + \epsilon\lambda'_{p\pm}(\delta) =: \lambda^{\app}_{p\pm}(\delta,\epsilon) \\
\Phi(\theta = \delta\epsilon,\epsilon,x,z) &\approx U_{p\pm}^0(x,z) + \epsilon U'_{p\pm}(\delta,x,z) =: U^{\app}_{p\pm}(\delta,\epsilon,x,z).
\end{aligned}\right.
\end{equation}
Note however that in contrast with the analysis of \cite{CNR13}, our parameter $\delta $ does not depend on $\epsilon$. Like many constants involved, it depends on $p$, the label of the double eigenvalue under consideration, which has been fixed. This expansion will be rigorously justified in Proposition~\ref{quasimodes_order_epsilon_estimate_quasimodes}.
Inserting \eqref{Ansatz-N} into \eqref{equation_eigenvalue_problem_S1}, and formally identifying terms of order $1=\epsilon^0$, we find that $U_{p\pm}^0$ solves the spectral problem for flat bottom with periodic boundary conditions:
\begin{equation}\label{quasimodes_order_epsilon_system_u_0}
\left\{
\begin{aligned}
& -\Delta U_{p\pm}^0 = 0 \qquad \text{in} \; S \\
& \partial_z U_{p\pm}^0 = 0 \quad \text{on } \{z = -1\}, \qquad
 \partial_z U_{p\pm}^0 = \lambda_{2p}^0(0) U_{p\pm}^0 \quad \text{on } \{z = 0\}.
\end{aligned}\right.
\end{equation}
From Section~\ref{sec:fond-plat}, 
 $U_{p\pm}^0 = \alpha^\pm_+ \Phi_p + \alpha^\pm_- \Phi_{-p}$, 
with $\alpha^\pm_+,\alpha^\pm_- \in \mathbb{R}$,
and
$\Phi_p(x,z):=\Phi_p(0,x,z)$ given in \eqref{phi-p} with $\theta=0$.

Identifying the terms of order $\epsilon$, we request that $U'_{p\pm}(\delta,x,z)$ solves:
\begin{equation}\label{quasimodes_order_epsilon_system_u_prime}
\left\{
\begin{aligned}
& -\Delta U'_{p\pm}(\delta) = \div(Q_1\nabla U_{p\pm}^0) + 2i\delta \partial_x U_{p\pm}^0 \qquad \text{in} \; S, \\
& \partial_z U'_{p\pm}(\delta) = -(Q_1\nabla U_{p\pm}^0)\cdot e_z \qquad \text{on} \; \{z = -1\}, \\
& \partial_z U'_{p\pm}(\delta) = \lambda_{2p}^0(0)U'_{p\pm}(\delta) -(Q_1\nabla U_{p\pm}^0)\cdot e_z 
+ \lambda'_{p\pm}(\delta) U_{p\pm}^0 \qquad \text{on} \; \{z = 0\},
\end{aligned}\right.
\end{equation}
where $Q_1$ is given in \eqref{Q1-Qk}.
Before proving the validity of the approximation, we identify the values of $\lambda'_{p\pm}(\delta)$ for which \eqref{quasimodes_order_epsilon_system_u_prime} has a solution.

\begin{proposition}\label{quasimodes_order_epsilon_nazarov_gap_formula}
The system \eqref{quasimodes_order_epsilon_system_u_prime} has a variationnal solution if and only if 
\begin{equation}\label{equation_valeurs_propres_M_p_ordre_epsilon}
M_p \begin{pmatrix}
\alpha^\pm_+ \\ \alpha^\pm_-
\end{pmatrix} = \lambda'_{p\pm}(\delta)\begin{pmatrix}
 \alpha^\pm_+ \\ \alpha^\pm_-
\end{pmatrix}
\end{equation}
with
\[
M_p = 
\begin{pmatrix}
K_p\delta & \widehat{b}_{2p}F_{2p} \\ \overline{\widehat{b}_{2p}}F_{2p} & -K_p\delta
\end{pmatrix},
\quad
F_{2p} = \left(\frac{p}{\cosh(p)}\right)^2, \quad
K_p = \frac{p}{\cosh(p)^2}\left(1+\frac{\sinh(2p)}{2p}\right).
\]
In this case, we have
\begin{equation}\label{expression_lambda_prime_opening_periodic_order_epsilon}
 \lambda_{p\pm}'(\delta) = \pm \left(K_p^2\delta^2 + F_{2p}^2 |\widehat{b}_{2p}|^2\right)^\frac{1}{2} .
\end{equation}
\end{proposition}

\begin{proof}
A solution $U'_{p\pm}(\delta)$ of \eqref{quasimodes_order_epsilon_system_u_prime} satisfies, for all $V \in H^1(S)$ 
\begin{multline}\label{def-L}
\int_S \nabla U'_{p\pm}(\delta)\cdot \nabla \overline{V} - \lambda_{2p}^0(0) \int_{\gamma} U'_{p\pm}(\delta)\overline{V} \\
= -\int_S Q_1 \nabla U_{p\pm}^0 \cdot \nabla \overline{V} + 2i\delta \int_S \partial_x U_{p\pm}^0\overline{V} 
+ \lambda'_{p\pm}(\delta) \int_\gamma U_{p\pm}^0 \overline{V} \;\;:= L(V).
\end{multline}
By Riesz representation theorem, there exists a unique function $W \in H^1(S)$ such that 
\begin{equation*}
L(V) = (W,V)_{H^1(S)} = \int_S \nabla W \cdot \nabla \overline{V} + \int_\gamma W\overline{V}. 
\end{equation*}
Therefore
\begin{equation*}
\int_S \nabla (U'_{p\pm}(\delta)-W)\cdot \nabla \overline{V} + \int_{\gamma} (U'_{p\pm}(\delta)-W)\overline{V} = \frac{1}{\tau_{2p}^0(0)}\int_\gamma U'_{p\pm}(\delta)\overline{V}
\end{equation*}
with $\tau_{2p}^0(0) = (1+ \lambda_{2p}^0(0))^{-1}$, or equivalently,
\begin{equation*}
 a^{R,S}_{0,0}(U'_{p\pm}(\delta)-W,V) = \frac{1}{\tau_{2p}^0(0)}\int_\gamma U'_{p\pm}(\delta)\overline{V},
\end{equation*}
where $a_{0,0}^{R,S}$ is the hermitian form defined by \eqref{definition_a_S}. By Remark~\ref{remark_variational_formulation_S} and Proposition~\ref{proposition_formulation_resolvant_S}, the previous inequality implies that $U'_{p\pm}(\delta)-W$ is the solution of the elliptic problem \eqref{equation_resolvant_S} (when $\theta=\epsilon=0$) for $\xi=\frac{1}{\tau_{2p}^0(0)}\xi'_{p\pm}(\delta):= \frac{1}{\tau_{2p}^0(0)}U'_{p\pm}(\delta)_{\vert_{z=0}}$, hence
\[
\frac{1}{\tau_{2p}^0(0)}(1+G_0[0])^{-1} \xi'_{p\pm}(\delta)=(U'_{p\pm}(\delta)-W)_{\vert_{z=0}}
\]
and
\begin{equation}\label{spectral_problem_xi_prime_order_epsilon}
 \big((1+G_0[0])^{-1} - \tau_{2p}^0(0)\big)\xi'_{p\pm}(\delta) = -\tau_{2p}^0(0) W_{\vert_{z=0}}.
\end{equation}
From \eqref{ortho-0}, Equation \eqref{spectral_problem_xi_prime_order_epsilon} has a solution if and only if
\[
\int_{\T_{2\pi}} W(x,0)e^{\pm ip x} \dd x = 0. 
\]
Using that $\Phi_{\pm p}$ satisfies \eqref{quasimodes_order_epsilon_system_u_0}, we get that 
\begin{equation*}
L(\Phi_{\pm p}) = \int_S \nabla W \cdot \nabla \overline{\Phi_{\pm p}} + \int_\gamma W \overline{\Phi_{\pm p}} = (1+\lambda_{2p}^0(0))\int_\gamma W \overline{\Phi_{\pm p}} 
\end{equation*}
and therefore \eqref{spectral_problem_xi_prime_order_epsilon} has a solution if and only if $L(\Phi_{\pm p}) = 0$, that is
\begin{equation}\label{opening_order_epsilon_equation_lambda_prime}
\begin{aligned}
\lambda'_{p\pm}(\delta)\int_{\gamma} U^0_{p\pm} \overline{\Phi_{p}} &= \int_S Q_1 \nabla U^0_{p\pm} \cdot \nabla \overline{\Phi_{ p}} + 2i\delta \int_S U_{p\pm}^0\overline{\partial_x \Phi_{ p}} , \\
\lambda'_{p\pm}(\delta)\int_{\gamma} U^0_{p\pm} \overline{\Phi_{- p}} &= \int_S Q_1 \nabla U^0_{p\pm} \cdot \nabla \overline{\Phi_{- p}} + 2i\delta \int_S U_{p\pm}^0\overline{\partial_x \Phi_{- p}} .
\end{aligned}
\end{equation}
We have
\begin{equation}\label{coeffs_orthogonality_conditions_U_prime_1}
\int_\gamma U^0_{p\pm} \overline{\Phi_{p}}
= \int_\gamma (\alpha^\pm_+ \Phi_p + \alpha^\pm_- \Phi_{-p})\overline{\Phi_{p}} = 2\pi \alpha^\pm_+ \ , \quad 
\int_\gamma U^0_{p\pm} \overline{\Phi_{-p}} = 2\pi \alpha^\pm_- \ .
\end{equation}
To compute the first term in the right-hand sides of \eqref{opening_order_epsilon_equation_lambda_prime}, we use the following lemma whose proof is given in Appendix~\ref{app-A}.

\begin{lemma} \label{calcul-des-3-integrales}
\begin{align}
\int_S Q_1 \nabla \Phi_p \cdot \nabla \overline{\Phi_p} &= \int_S Q_1 \nabla \Phi_{-p} \cdot \nabla \overline{\Phi_{-p}} = 0 \\
\int_S Q_1 \nabla \Phi_p \cdot \nabla \overline{\Phi_{-p}} &= 2\pi\left(\frac{p}{\cosh(p)}\right)^2\overline{\widehat{b}_{2p}}, \label{integral_Q1_nabla_U_0_nabla_phi_q}\\
\int_S Q_1 \nabla \Phi_{-p} \cdot \nabla \overline{\Phi_{p}} &= 2\pi\left(\frac{p}{\cosh(p)}\right)^2\widehat{b}_{2p}.\label{integral_Q1_nabla_U_0_nabla_phi_q2}
\end{align}
\end{lemma}

We are left to compute
\begin{equation}\label{coeffs_orthogonality_conditions_U_prime_2}
\begin{aligned}
2i\delta \int_S U_{p\pm}^0\overline{\partial_x \Phi_{p}}
&= 2i(-ip)\delta \int_S(\alpha^\pm_+ \Phi_p + \alpha^\pm_- \Phi_{-p})\overline{\Phi_p} \\
 &= 4\pi p \delta \alpha^\pm_+ \int_{-1}^0 \frac{\cosh(p(z+1))^2}{\cosh(p)^2}\dd z \\
&= 2\pi \delta \alpha^\pm_+\frac{p}{\cosh(p)^2}\int_{-1}^0 (1+\cosh(2p(z+1))\dd z \\
&= 2\pi \delta \alpha^\pm_+\frac{p}{\cosh(p)^2}\left(1+\frac{\sinh(2p)}{2p}\right)
\end{aligned}
\end{equation}
and
\begin{equation}
\label{coeffs_orthogonality_conditions_U_prime_3}
\begin{aligned}
2i\delta \int_S U_{p\pm}^0\overline{\partial_x \Phi_{-p}}
= -2\pi \delta \alpha^\pm_-\frac{p}{\cosh(p)^2}\left(1+\frac{\sinh(2p)}{2p}\right).
\end{aligned}
\end{equation}
Using \eqref{coeffs_orthogonality_conditions_U_prime_1}, \eqref{integral_Q1_nabla_U_0_nabla_phi_q}, \eqref{coeffs_orthogonality_conditions_U_prime_2} and \eqref{coeffs_orthogonality_conditions_U_prime_3}, we can rewrite \eqref{opening_order_epsilon_equation_lambda_prime} as \eqref{equation_valeurs_propres_M_p_ordre_epsilon}. We find the value of $\lambda_{p\pm}'(\delta)$ given by Proposition~\ref{quasimodes_order_epsilon_nazarov_gap_formula} by computing the eigenvalues of $M_p$. 
\end{proof}

Denote 
$\xi_{p\pm}^{\app} := U^{\app}_{p\pm}{\vert_{z=0}}, \;\;
\tau_{p\pm}^{\app} := (1+\lambda_{p\pm}^{\app})^{-1}$.
The next result shows that $(\tau_{p\pm}^{\app},\xi_{p\pm}^{\app})$ are approximate eigenpairs of the resolvent operator.

\begin{proposition}\label{quasimodes_order_epsilon_estimate_quasimodes}
Under the assumptions of Proposition~\ref{quasimodes_order_epsilon_nazarov_gap_formula}, and assuming $\epsilon$ is small enough (depending only on $b$ and $p$), we have for all $\delta\in [0,d_{p,1}]$
\begin{equation*}
 \norm{(1+G_{\delta \epsilon}[\epsilon b])^{-1}\xi_{p\pm}^{\app} - \tau_{p\pm}^{\app}\xi_{p\pm}^{\app}}_{L^2(\T_{2\pi})} \leq C_{b,p} \epsilon^2\norm{\xi_{p\pm}^{\app}}_{L^2(\T_{2\pi})}.
\end{equation*}
\end{proposition}

\begin{proof}
Inserting $U^{\app}_{p\pm} =U_{p\pm}^0 + \epsilon U'_{p\pm}$ in \eqref{definition_a_S}, we find that, for $V \in H^1(S)$, 
\begin{align*}
a^{R,S}_{\delta\epsilon,\epsilon}(U^{\app}_{p\pm} ,V) =&
\int_S P(\Sigma)\nabla U^{\app}_{p\pm} \cdot \nabla \overline{V} + \delta^2\epsilon^2 \int_S (1-\epsilon b(x)) U^{\app}_{p\pm} \overline{V} \\
&+ \int_\gamma U^{\app}_{p\pm} \overline{V} +i\delta\epsilon \int_{S} \left(e_1 +\epsilon\begin{pmatrix}
-b(x) \\
z b'(x)
\end{pmatrix}\right) \cdot( U^{\app}_{p\pm} \nabla \overline{V} -\overline{V}\nabla U^{\app}_{q\pm} ) \\
=& (\lambda_{p\pm}^{\app}+1)\int_\gamma U^{\app}_{p\pm}\overline{V} + \widetilde{L}(V)
\end{align*}
where
\begin{align*}
\widetilde{L}(V) =& -\epsilon^2\lambda_p'(\delta)\int_\gamma U'_{p\pm}(\delta)\overline{V} +\epsilon^2 \int_S \widetilde{Q}\nabla U^{\app}_{p\pm}\cdot \nabla \overline{V} + \epsilon^2\int_S Q_1 \nabla U'_{p\pm}(\delta)\nabla \overline{V} 
\\
&+\delta^2\epsilon^2\int_S(1-\epsilon b(x)) U^{\app}_{p\pm}\overline{V} 
+i\delta\epsilon^2 \int_S \begin{pmatrix}
-b(x) \\
z b'(x)
\end{pmatrix}\cdot( U^{\app}_{p\pm}\nabla\overline{V}-\overline{V} \nabla U^{\app}_{p\pm} ) \\
&+ i\epsilon^2\delta\int_S ( U'_{p\pm}(\delta)\partial_x \overline{V} - \overline{V}\partial_x U'_{p\pm}(\delta)). 
\end{align*}
 and $\epsilon^2\widetilde{Q} =P(\Sigma)-I_2 -\epsilon Q_1 $.
By Lax-Milgram, there exists $\widetilde{W}$ such that $a_{\delta\epsilon,\epsilon}^{R,S}(\widetilde{W},\cdot) = \widetilde{L} (\cdot) $. Using that $\tau_{p\pm}^{\app} = (1+\lambda_{p\pm}^{\app})^{-1}$, we get 
\begin{equation}\label{formulation_variationnelle_vp_approche_ordre_epsilon}
a_{\delta\epsilon,\epsilon}^{R,S}(U^{\app}_{p\pm} - \widetilde{W},V) = \frac{1}{\tau_{p\pm}^{\app}}\int_\gamma U^{\app}_{p\pm} \overline{V}. 
\end{equation}
From Remark~\ref{remark_variational_formulation_S} and Proposition~\ref{proposition_formulation_resolvant_S}, the previous inequality implies that $U^{\app}_{p\pm} - \widetilde{W}$ is the solution of the elliptic problem \eqref{equation_resolvant_S} (when $\theta=\delta\epsilon$) for $\xi=\frac{1}{\tau_{p\pm}^{\app}}\xi_{p\pm}^{\app} = \frac{1}{\tau_{p\pm}^{\app}}U^{\app}_{p\pm}\vert_{z=0}$, hence
\[
\frac{1}{\tau_{p\pm}^{\app}}(1+G_{\delta\varepsilon}[\varepsilon b])^{-1} \xi_{p\pm}^{\app} =(U^{\app}_{p\pm} - \widetilde{W})\vert _{z=0}
\]
thus, 
\begin{equation}\label{norme_quasimode_non_normalise_ordre_epsilon}
\norm{(1+G_{\delta\epsilon}[\epsilon b])^{-1}\xi_{p\pm}^{\app} - \tau_{p\pm}^{\app}\xi_{p\pm}^{\app}}_{L^2} \leq C_p\|\widetilde{W}\|_{H^1(S)} \leq C_p\|\widetilde{L}\|_{H^1(S)'},
\end{equation}
with
\begin{equation}\label{operateur-L}
 \|\widetilde{L}\|_{H^1(S)'} \leq C_b \epsilon^2\big(\norm{U_{p\pm}^0}_{H^1(S)}+\norm{U'_{p\pm}(\delta)}_{H^1(S)}\big).
\end{equation}
Using \eqref{spectral_problem_xi_prime_order_epsilon}, we may choose $U'_{p\pm} (\delta)_{\vert_{z=0}}= \xi'_{p\pm}(\delta) = - \tau_{2p}^0(0) R_{2p,0}^{-1}(W_{\vert_{z=0}})$ where $R_{2p,0}^{-1}$ is the operator defined in \eqref{operateur-R_pinv}. Therefore 
\begin{equation*}
\norm{U'_{p\pm}(\delta)}_{H^1(S)} \leq C_p\|W\|_{H^1(S)} \leq C_{b,p}\norm{L}_{H^1(S)'} \leq C_{b,p}\norm{U_{p\pm}^0}_{H^1(S)} ,
\end{equation*}
where $L$ is defined in \eqref{def-L}, and it follows that
\begin{equation*}
 \|\widetilde{L}\|_{H^1(S)'} \leq C_{b,p}\epsilon^2\norm{U_{p\pm}^0}_{H^1(S)}
\end{equation*}
and if $\epsilon$ is small enough (depending on $p$ and $b$)
\begin{equation}\label{borne_inf_norme_U_app_ordre_epsilon}
 \norm{U^{\app}_{p\pm}}_{H^1(S)} \geq \norm{U_{p\pm}^0}_{H^1(S)}(1 - C_{b,p}\epsilon)
 \geq \frac{1}{2}\norm{U_{p\pm}^0}_{H^1(S)}.
\end{equation}
From \eqref{formulation_variationnelle_vp_approche_ordre_epsilon} and the coercivity of $a^{R,S}_{\theta,\epsilon}$ independently of $\theta$ and $\epsilon$, we get that if $\epsilon$ is small enough,
\begin{equation*}
 \norm{\xi^{\app}_{p\pm}}_{L^2(\T_{2\pi})} \geq c_{b,p}(\norm{U^{\app}_{p\pm}}_{H^1(S)} - \|\widetilde{W}\|_{H^1(S)}) \geq c_{b,p}\norm{U^0_{p\pm}}_{H^1(S)}.
\end{equation*}
Estimate of Proposition~\ref{quasimodes_order_epsilon_estimate_quasimodes} results from combining the above equation with \eqref{operateur-L} and \eqref{norme_quasimode_non_normalise_ordre_epsilon}.
\end{proof}

\begin{proof}[Proof of Theorem~\ref{one_gap-order-eps}]
Let $\delta \in [0,d_{1,p}]$, $\theta = \delta\epsilon \in [0,d_{1,p}\epsilon]$, and $ u^{\app}_{p\pm}= \xi^{\app}_{p\pm}/ \| \xi^{\app}_{p\pm}\| _{L^2(\T_{2\pi})} $.
We now apply Lemma~\ref{lemma:quasimodes} for operator $K_\theta(\epsilon)= (1+G_\theta[\epsilon b] )^{-1}$ with pairs $( \tau_{p+}^{\app}(\delta,\epsilon), u^{\app}_{p+})$ and $(\tau_{p-}^{\app}(\delta,\epsilon), u^{\app}_{p-})$.

If $\epsilon < \epsilon_p$ (for some $\epsilon_p$ depending only on $b$ and $p$), there exist two eigenvalues $\tau_{p\pm}^\epsilon(\theta)$ of $(1+G_\theta[\epsilon b])^{-1}$ such that $|\tau_{p\pm}^\epsilon(\theta) - \tau_{p\pm}^{\app}(\theta\epsilon^{-1},\epsilon)| \leq C_p\epsilon^{2}$. 
Consequently, there exist two eigenvalues $\lambda_{p\pm}^\epsilon(\theta)$ of $G_\theta[\epsilon b]$ such that $|\lambda_{p\pm}^\epsilon(\theta) - \lambda_{p\pm}^{\app}(\theta\epsilon^{-1},\epsilon)| \leq C_p\epsilon^{2}$. Using the expression \eqref{expression_lambda_prime_opening_periodic_order_epsilon} for $\lambda_{p\pm}'(\delta)$, we get 
\[|\lambda_{p\pm}^\epsilon(\theta) - \lambda_{2p}^0(0)| \leq C_p \epsilon^{2} +C_p\epsilon \leq 2C_p\epsilon<\frac18,
\]
for $\epsilon \leq \epsilon_{p}$.
By Proposition~\ref{double-ev0}, the spectrum of $G_\theta[\epsilon b]$ has exactly two eigenvalues in a neighborhood of $\lambda^0_{2p}(0)$, therefore $\lambda_{2p-1}^\epsilon(\theta)=\lambda_{p-}^\epsilon(\theta)$ and $\lambda_{2p}^\epsilon(\theta)=\lambda_{p+}^\epsilon(\theta)$.
Note that
\begin{align*}
 \lambda_{2p}^\epsilon(\theta) \geq& \lambda^{\app}_{p+}(\delta,\epsilon) - |\lambda_{p+}^\epsilon(\theta)-\lambda^{\app}_{p+}(\theta\epsilon^{-1},\epsilon)| \\
 \geq& \lambda_{2p}^0(0) + \epsilon\lambda_{p+}'(\delta) - C_p\epsilon^{2} 
 \geq \lambda_{2p}^0(0) + F_{2p}|\widehat{b}_{2p}|\epsilon - C_p\epsilon^{2}.
\end{align*}
and similarly, 
\begin{equation*}
\lambda_{2p-1}^\epsilon(\theta) \leq \lambda_{2p}^0(0) -
F_{2p}|\widehat{b}_{2p}|\epsilon + C_p\epsilon^{2}.
\end{equation*}
Thus, for $0\leq\theta \leq d_{p,1} \epsilon$, we have obtained a lower bound for the separation between the two eigenvalues of $\sigma(G_\theta[\epsilon b])$ in the vicinity of $\theta=0$, if $\widehat{b}_{2p}\neq 0$:
\begin{equation*}
 \lambda_{2p}^\epsilon(\theta) - \lambda_{2p-1}^\epsilon(\theta) \geq 2F_{2p}|\widehat{b}_{2p}|\epsilon -C_p \epsilon^{2}.
\end{equation*}
Let us note that it is the first time in this section that we need $\widehat{b}_{2p}\neq 0$.

A last step is needed to show that the gap remains open for $\theta$ far from $0$. For this, we return to estimates for the perturbation of simple eigenvalues.
In Proposition~\ref{simple-ev}, we proved that, for sufficiently small $\epsilon$, and $ d_{p,1} \epsilon \le \theta \le \frac12- d_{p,1} \epsilon$, $\lambda_{2p}^\epsilon(\theta)$ is larger than $\lambda_{2p}^0(0) + F_{2p}|\widehat{b}_{2p}|\epsilon$ and $\lambda_{2p-1}^\epsilon(\theta)$ is smaller than $\lambda_{2p}^0(0) - F_{2p}|\widehat{b}_{2p}|\epsilon$.

Finally, to find the precise size of the gap, we take $\theta$ and $\delta$ very small, say $\delta\leq \epsilon$ hence $\theta =\delta \epsilon \le \epsilon^2$. From \eqref{expression_lambda_prime_opening_periodic_order_epsilon}, 
\begin{equation*}
 \lambda_{p\pm}'(\delta) =\pm F_{2p}|\widehat{b}_{2p}| + \mathcal{O}(\epsilon), \;\;
 \lambda_{p\pm}^{\app} = \lambda_{2p}^0(0) \pm \epsilon F_{2p}|\widehat{b}_{2p}| + \mathcal{O}(\epsilon^2)
\end{equation*}
and
\begin{equation*}
 |\lambda_{p\pm}^\epsilon (\theta) - (\lambda_{2p}^0(0) \pm F_{2p}|\widehat{b}_{2p}|\epsilon)| \leq C_p\epsilon^{2}.
\end{equation*}
This gives the precise size of the gap opening, centered at $\lambda_{2p}^0(0)$ of length $2\epsilon F_{2p}|\widehat{b}_{2p}|$ plus small corrections.

For the proof of Theorem~\ref{one_gap-order-eps} (ii), setting $\Phi_\theta := e^{i\theta x}\Phi$, we observe that the spectral problem \eqref{equation_eigenvalue_problem_S1} can be written
\begin{equation*}
\int_S P(\Sigma)\nabla \Phi_\theta \cdot \nabla \overline{V} = \lambda(\theta,\epsilon)\int_\gamma \Phi_\theta \overline{V} \; \quad \text{for all } V \in H^1_\theta(S). 
\end{equation*}
where $\Phi_\theta \in H^1_\theta(S)$, the space of $\theta$-periodic functions in $x$ (that is $\Phi_\theta\vert_{x=2\pi} = e^{2i\pi\theta}\Phi_\theta\vert_{x=0}$). 

Set $\beta = \frac{1}{2}-\theta$ and $\widetilde{\Phi}_\theta := e^{i\beta x}\Phi_\theta $ then $\widetilde{\Phi}_\theta \in H_{\frac{1}{2}}^1(S)$ (i.e $\widetilde{\Phi}_\theta$ antiperiodic). As we did in Proposition~\ref{proposition_formulation_resolvant_S} and Remark~\ref{remark_variational_formulation_S}, we can show that the spectral problem above is equivalent to
\begin{align*}
&\int_S \Big[P(\Sigma)\nabla \widetilde{\Phi}_\theta \cdot \nabla \overline{V} 
-i\beta \Big(e_1 +\epsilon\begin{pmatrix}-b(x) \\z b'(x)\end{pmatrix}\Big) 
\cdot( \widetilde{\Phi}_\theta \nabla \overline{V}-\overline{V} \nabla \widetilde{\Phi}_\theta)\Big]\\
& + \int_S \beta^2 (1-\epsilon b)\widetilde{\Phi}_\theta \overline{V} = \lambda\big(\tfrac{1}{2}-\beta,\epsilon\big)\int_{\gamma} \widetilde{\Phi}_\theta \overline{V} \quad \text{for all} \; V \in H^1_{\frac{1}{2}}(S). 
\end{align*}
Then, as for the periodic case, we use the Ansatz $\beta = \delta \epsilon$ and
\begin{equation*} 
\left\{
\begin{aligned}
\lambda_{p\pm}(\beta = \delta\epsilon,\epsilon) &\approx \lambda_{2p}^0\big(\tfrac{1}{2}\big) + \epsilon\lambda'_{p\pm}(\delta) =: \lambda^{\app}_{p\pm}(\delta,\epsilon) \\
\widetilde{\Phi}_\theta(\beta = \delta\epsilon,\epsilon,x,z) &\approx \widetilde{U}_{p\pm}^0(x,z) + \epsilon \widetilde{U}'_{p\pm}(\delta,x,z) =: \widetilde{U}^{\app}_{p\pm}(\delta,\epsilon,x,z).
\end{aligned}\right.
\end{equation*}
where $\widetilde{U}_{p\pm}^0 = \alpha^\pm_+ \Psi_p + \alpha^\pm_- \overline{\Psi_{p}}$, 
with $\alpha^\pm_+,\alpha^\pm_- \in \mathbb{R}$ and
\begin{align*}
 \Psi_p(x,z) = \frac{\cosh\big(\big(p+\frac{1}{2}\big)(z+1)\big)}{\cosh\big(p+\frac{1}{2}\big)}e^{i\big(p+\frac{1}{2}\big)x}
\end{align*}
and its conjugate are the eigenvectors associated to $\lambda_{2p}^0\big(\frac{1}{2}\big) = \lambda_{2p+1}^0\big(\frac{1}{2}\big)$. The rest of the proof proceeds as in the periodic case with very similar computations.
\end{proof}

\section{Gap opening of order $\epsilon^2$}\label{section:5}

When $\widehat{b}_{2p} =0$, a higher order asymptotic expansion for the study of eigenvalues close to $\theta=0$ is required to open a gap. We will construct an expansion valid for $\theta \in [0,M\epsilon^{2}]$ for any $M>0$. In order to show that the gap does not close for $ \theta \in [ M\epsilon^{2}, d_{p,1}\epsilon]$, (a region not covered in Proposition~\ref{simple-ev}), we use the end of Section~\ref{section:4} in the special case $\widehat{b}_{2p} =0$ which will be sufficient to control the separation of eigenvalues in this region. As it was made in Proposition~\ref{simple-ev} with the condition on $d_{p,2}$, we will need at the end of this section that the eigenvalues are enough separated. For this purpose, we set
\[
M_p : = 1 +|J_p(b)|+|S_p(b)|
\]
where $J_p(b)$ and $S_p(b)$ are defined in Theorem~\ref{one_gap-order-eps2}.

\begin{proposition}[Perturbation of a simple eigenvalue] \label{single-ev-higher-order}
Fix $p \in \mathbb{N}^\ast$ and assume $\widehat{b}_{2p} =0$. There exist $\epsilon_{p,5}>0$ and $d_{p,5}>0$ depending on $p$ and $b$, such that, for all $\epsilon \in (0,\epsilon_{p,5})$, and $ d_{p,5}\epsilon^{2} \le \theta \le d_{p,1}\epsilon$, we have
\begin{equation*}
 \sigma(G_\theta[\epsilon b]) \cap \Big[\lambda_{2p-1}^{0} (\tfrac12) , \lambda_{2p}^{0} (\tfrac12)\Big] = \{\lambda_{2p-1}^{\epsilon} (\theta), \lambda_{2p}^{\epsilon} (\theta) \}
\end{equation*}
where $\lambda_{2p-1}^{\epsilon} (\theta)$ and $\lambda_{2p}^{\epsilon} (\theta)$ are simple.
Moreover,
\begin{equation*}
 \lambda_{2p-1}^{\epsilon} (\theta) \leq\lambda_{2p}^{0} (0) - M_p \epsilon^2, \quad
 \lambda_{2p}^{\epsilon} (\theta) \geq\lambda_{2p}^{0} (0) + M_p \epsilon^2.
\end{equation*}
\end{proposition}

\begin{proof}
As noted at the end of Section~\ref{section:4}, we can use the analysis developed therein to state
\begin{align*}
 \lambda_{2p}^\epsilon(\theta) \geq& \lambda^{\app}_{p+}(\theta,\epsilon) - |\lambda_{p+}^\epsilon(\theta)-\lambda^{\app}_{p+}(\theta\epsilon^{-1},\epsilon)| \\
 \geq& \lambda_{2p}^0(0) + \epsilon\lambda_{p+}'(\delta) - C_p\epsilon^{2} 
 = \lambda_{2p}^0(0) + K_p\theta - C_p\epsilon^{2}.
\end{align*}
for $\theta\in [0,d_{p,1}\epsilon]$, where we have now used $\widehat{b}_{2p} =0$ in the expression \eqref{expression_lambda_prime_opening_periodic_order_epsilon} of $\lambda_{p\pm}'(\delta)$, and similarly, 
\begin{equation*}
\lambda_{2p-1}^\epsilon(\theta) \leq \lambda_{2p}^0(0) - K_p\theta+ C_p\epsilon^{2}.
\end{equation*}
For $\theta \geq d_{p,5} \epsilon^2$, we conclude the proof of this proposition by choosing $d_{p,5} := (M_p+C_p)/K_p$.
\end{proof}

Thanks to the previous proposition, we only need to construct approximate solutions of \eqref{equation_eigenvalue_problem_S2} for $\theta \in [0,d_{p,5}\epsilon^{2}]$ (instead of $[0,d_{p,1}\epsilon]$ as in Section~\ref{section:4}) and we therefore write the following Ansatz: 
\begin{equation*}
\left\{
\begin{aligned}
\lambda_{p\pm}(\theta = \delta\epsilon^2,\epsilon) &\approx \lambda_{2p}^0(0) + \epsilon^2\lambda''_{p\pm}(\delta) =: \lambda^{\app}_{p\pm}(\delta,\epsilon) \\
 \Phi(\theta = \delta\epsilon^2,\epsilon,x,z) &\approx U_{p\pm}^0(x,z) + \epsilon U'_{p\pm}(x,z) + \epsilon^2 U''_{p\pm}(\delta,x,z) =: U^{\app}_{p\pm}(\delta,\epsilon,x,z).
\end{aligned}\right.
\end{equation*}
As in Section~\ref{section:4}, if we insert this Ansatz in \eqref{equation_eigenvalue_problem_S1} and formally identify the terms of order $1$, we find that $U_{p\pm}^0$ solves the spectral problem for flat bottom with periodic boundary conditions \eqref{quasimodes_order_epsilon_system_u_0} and therefore $U_{p\pm}^0 = \alpha^\pm_+ \Phi_p + \alpha^\pm_- \Phi_{-p}$, with $\alpha^\pm_+,\alpha^\pm_- \in \mathbb{R}$ and $\Phi_p(x,z):=\Phi_p(0,x,z)$ given in \eqref{phi-p} with $\theta=0$.

Identifying the terms of order $\epsilon$, we request that $U_{p\pm}'$ solves
\begin{equation}\label{systeme_U_prime_ordre_2}
\left\{
\begin{aligned}
& -\Delta U_{p\pm}' = \div(Q_1\nabla U_{p\pm}^0) \qquad \textrm{in} \; S \\
& \partial_z U_{p\pm}' = - Q_1 \nabla U_{p\pm}^0 \cdot e_z \qquad \textrm{on} \; \{z = -1\} \\
& \partial_z U_{p\pm}' = \lambda_{2p}^0(0) U_{p\pm}' - Q_1 \nabla U_{p\pm}^0 \cdot e_z \qquad \textrm{on} \; \{z = 0\}.
\end{aligned}\right.
\end{equation}
Using \eqref{ortho-0} as we did to obtain \eqref{opening_order_epsilon_equation_lambda_prime}, we get that this system has a solution if and only if
\[
0 = \int_S Q_1 U_{p\pm}^0 \cdot \nabla \overline{\Phi_{p}} =\int_S Q_1 U_{p\pm}^0 \cdot \nabla \overline{\Phi_{-p}} .
\]
Since $\widehat{b}_{2p} = 0$, Lemma~\ref{calcul-des-3-integrales} implies that the above condition is always satisfied. The next lemma provides an explicit formula for the solution $U_{p\pm}'$. The details of the computations are given in Appendix~\ref{app-B}.

\begin{lemma}\label{formula_U_prime_ordre_2}
A particular solution of \eqref{systeme_U_prime_ordre_2} is given by
\begin{equation*}
U'_{p \pm}(x,z) = E_\pm(x,z) + \sum_{k \notin \{0,p,-p\}} (\beta_k \cosh(k(z+1)) + \gamma_k\sinh(k(z+1)))e^{ikx}
\end{equation*}
where
\begin{equation*}
\gamma_k = \frac{p}{\cosh(p)}(-\alpha_+^{\pm}\widehat{b}_{k-p}+\alpha_-^{\pm}\widehat{b}_{k+p}), \quad
\beta_k = \frac{k^2 - \kappa_k(0)\kappa_p(0)}{k(\kappa_p(0) - \kappa_k(0))}\gamma_k,
\end{equation*}
and $E_\pm = \alpha^{\pm}_+ E_p + \alpha^{\pm}_- E_{-p}$, 
with
\begin{equation} \label{fonction-auxil-G}
E_{\pm p} := -\frac{p}{\cosh(p)}z \sinh(p(z+1)) b(x)e^{\pm ipx}.
\end{equation}
\end{lemma}

Now, we identify the terms of order $\epsilon^2$ and request that $U_{p\pm}''$ solves
\begin{equation*}
\left\{
\begin{aligned}
& -\Delta U_{p\pm}'' = \div(Q_1\nabla U_{p\pm}') + \div(Q_2\nabla U_{p\pm}^0) + 2i\delta \partial_x U_{p\pm}^0 \qquad \textrm{in} \; S, \\
& \partial_z U_{p\pm}'' = - (Q_1 \nabla U_{p\pm}')\cdot e_z - (Q_2 \nabla U_{p\pm}^0 )\cdot e_z \qquad \textrm{on} \; \{z = -1\}, \\
& \partial_z U_{p\pm}'' = \lambda_{2p}^0(0) U_{p\pm}'' - (Q_1 \nabla U_{p\pm}')\cdot e_z - (Q_2 \nabla U_{p\pm}^0 )\cdot e_z + \lambda''_{p\pm}(\delta)U_{p\pm}^0 \; \textrm{on} \; \{z = 0\}.
\end{aligned}\right.
\end{equation*}
The orthogonality conditions to solve this system are similar to \eqref{opening_order_epsilon_equation_lambda_prime} and take the form:
\begin{equation}\label{orthogonality_condition_lambda''}
\begin{aligned}
\lambda''_{p\pm}(\delta) \int_\gamma U_{p\pm}^0 \overline{\Phi_{p}} =& \int_S Q_1 \nabla U_{p\pm}'\cdot \nabla \overline{\Phi_{p}} + \int_S Q_2 \nabla U_{p\pm}^0 \cdot \nabla \overline{\Phi_{p}} \\&+ 2i\delta\int_S U_{p\pm}^0 \overline{ \partial_x \Phi_{p}},\\
\lambda''_{p\pm}(\delta) \int_\gamma U_{p\pm}^0 \overline{\Phi_{- p}} =& \int_S Q_1 \nabla U_{p\pm}'\cdot \nabla \overline{\Phi_{- p}} + \int_S Q_2 \nabla U_{p\pm}^0 \cdot \nabla \overline{\Phi_{- p}} \\
&+ 2i\delta\int_S U_{p\pm}^0 \overline{ \partial_x \Phi_{- p}}.
\end{aligned}
\end{equation}

We now compute the values of $\lambda''_{p\pm}(\delta)$ for which these conditions are satisfied.

\begin{proposition}
The orthogonality conditions \eqref{orthogonality_condition_lambda''} are satisfied if and only if 
\begin{equation}\label{equation_vp_lambda_second}
N_p \begin{pmatrix}
 \alpha_+^\pm \\ \alpha_-^\pm
\end{pmatrix} = \lambda''_{p\pm}(\delta)\begin{pmatrix}
 \alpha_+^\pm \\ \alpha_-^\pm
\end{pmatrix}
\end{equation}
where
\begin{equation*}
N_p = 
\begin{pmatrix}
 (K_p\delta + J_p(b)) & -S_p(b)\\ -\overline{S_p(b)} & (-K_p\delta + J_p(b))
\end{pmatrix}
\end{equation*}
\begin{align*}
J_p(b) &= \frac{p^2}{\cosh(p)^2}\sum_{k \notin \{0,p,-p\}}\frac{k^2-\kappa_k(0)\kappa_p(0)}{\kappa_p(0) - \kappa_k(0)}|\widehat{b}_{k-p}|^2 \\
S_p(b) &= \frac{p^2}{\cosh(p)^2}\sum_{k \notin \{0,p,-p\}}\frac{k^2-\kappa_k(0)\kappa_p(0)}{\kappa_p(0) - \kappa_k(0)}\widehat{b}_{k+p}\overline{\widehat{b}}_{k-p}.
\end{align*}
In this case we have
\begin{equation}\label{definition_lambda''} 
\lambda_{p\pm}''(\delta) = J_p(b) \pm \sqrt{K_p^2\delta^2 + |S_p(b)|^2}.
\end{equation}
\end{proposition}

\begin{proof}
The first term on the right-hand side of \eqref{orthogonality_condition_lambda''} takes the form
\[
\int_S Q_1 \nabla U_{p\pm}' \cdot \nabla \overline{\Phi_{\beta p} }
= \int_{\partial S} U_{p\pm}' (Q_1\nabla \overline{ \Phi_{\beta p}}) \cdot \vec{n} - \int_S U_{p\pm}'\div(Q_1\nabla \Phi_{-\beta p})
\]
for $\beta \in \{+,-\}$.
Using $\div(Q_1\nabla \Phi_{-\beta p})= - \Delta E_{-\beta p}$ (see \eqref{laplacien-E-p}), we write
 \begin{align*}
 - \int_S U_{p\pm}'\div(Q_1\nabla \Phi_{-\beta p})
 &= \int_S U_{p\pm}'\Delta E_{-\beta p}\\
 &=\int_S \Delta U_{p\pm}' E_{-\beta p} + \int_{\partial S}(U_{p\pm}'\partial_n E_{-\beta p} - E_{-\beta p}\partial_n U_{p\pm}')\\
 &=\int_S \Delta E_\pm E_{-\beta p} + \int_{\partial S}U_{p\pm}'\partial_n E_{-\beta p}
 \end{align*} 
since $E_{\beta p} = 0$ on $\partial S$.
 A calculation (see proof in Appendix~\ref{app-C}) shows that 
\begin{equation} \label{claim}
\int_S \Delta E_\pm E_{\beta p} =- \int_S Q_2 \nabla U_{p\pm}^0 \cdot \nabla \Phi_{\beta p}, 
\end{equation} 
thus, 
\begin{equation*}
\int_S Q_1 \nabla U_{p\pm}' \cdot \nabla \overline{\Phi_{\beta p} }
= \int_{\partial S} U_{p\pm}' (Q_1\nabla \overline{ \Phi_{\beta p}})\cdot \vec{n} - \int_S Q_2 \nabla U_{p\pm}^0 \cdot \nabla \overline{\Phi_{\beta p}} + \int_{\partial S}U_{p\pm}'\partial_n E_{-\beta p}.
 \end{equation*}
We now return to orthogonality condition \eqref{orthogonality_condition_lambda''} which becomes, 
\begin{align*}
\lambda''_{p\pm}(\delta) \int_\gamma U_{p\pm}^0 \overline{\Phi_{\beta p}} &= \int_{\partial S} U'_{p\pm} (Q_1\nabla \overline{\Phi_{\beta p}})\cdot \vec{n} + \int_{\partial S}U'_{p\pm}\partial_n E_{-\beta p} + 2i\delta\int_S U_{\pm p}^0 \partial_x \overline{\Phi_{\beta p}} \\
&= -\int_{\{z=-1\}}U'_{p\pm}\Big((Q_1\nabla \overline{\Phi_{\beta p}})\cdot e_z + \partial_z E_{-\beta p}\Big) + 2i\delta\int_S U_{\pm p}^0 \partial_x \overline{\Phi_{\beta p}},
\end{align*}
where we have used $-(Q_1\nabla \overline{\Phi_{\beta p}})\cdot e_z= -\kappa_p(0)b(x)e^{-i\beta px} = \partial_z E_{-\beta p}$ on $\{z=0\}$ see \eqref{B2}-\eqref{B3}.
With \eqref{B2}-\eqref{B3}, we compute the boundary term for $\beta= +$ and $\beta=-$ respectively
\begin{align*}
-&\int_{\{z=-1\}} U'_{p\pm}(Q_1 \nabla \Phi_{-p} \cdot e_z + \partial_z E_{-p}) \\ 
=&- \sum_{k \notin \{0,p,-p\}} \beta_k\int_0^{2\pi}\left(\frac{ip}{\cosh(p)}b'(x)+\frac{p^2}{\cosh(p)}b(x)\right) e^{i(k-p)x} \dd x \\
=&-2\pi \sum_{k \notin \{0,p,-p\}} \beta_k\frac{kp}{\cosh(p)}\overline{\widehat{b}_{k-p}} \\
=& 2\pi\sum_{k \notin \{0,p,-p\}}\frac{k^2-\kappa_k(0)\kappa_p(0)}{k(\kappa_p(0) - \kappa_k(0))}\frac{p}{\cosh(p)}(\alpha_+^{\pm}\widehat{b}_{k-p}-\alpha_-^{\pm}\widehat{b}_{k+p})\frac{kp}{\cosh(p)}\overline{\widehat{b}_{k-p}} \\
=& \alpha_+^{\pm}\frac{2\pi p^2}{\cosh(p)^2}\sum_{k \notin \{0,p,-p\}}\frac{k^2-\kappa_k(0)\kappa_p(0)}{\kappa_p(0)- \kappa_k(0)}|\widehat{b}_{k-p}|^2 \\
&-\alpha_-^{\pm}\frac{2\pi p^2}{\cosh(p)^2}\sum_{k \notin \{0,p,-p\}}\frac{k^2-\kappa_k(0)\kappa_p(0)}{\kappa_p(0) - \kappa_k(0)}\widehat{b}_{k+p}\overline{\widehat{b}_{k-p}}\\
=&2\pi J_p(b)\alpha_+^{\pm}-2\pi S_p(b)\alpha_-^{\pm}
\end{align*}
and
\begin{align*}
-&\int_{\{z=-1\}} U'_{p\pm}(Q_1 \nabla \Phi_p \cdot e_z + \partial_z E_p) \\
=& \sum_{k \notin \{0,p,-p\}} \beta_k\int\left(\frac{ip}{\cosh(p)}b'(x)-\frac{p^2}{\cosh(p)}b(x)\right) e^{i(k+p)x} \dd x \\
=& 2\pi\sum_{k \notin \{0,p,-p\}} \beta_k\frac{kp}{\cosh(p)}\overline{\widehat{b}_{k+p}} \\
=& 2\pi \sum_{k \notin \{0,p,-p\}}\frac{k^2-\kappa_k(0)\kappa_p(0)}{k(\kappa_p(0) - \kappa_k(0))}\frac{p}{\cosh(p)}(-\alpha_+^{\pm}\widehat{b}_{k-p}+\alpha_-^{\pm}\widehat{b}_{k+p})\frac{kp}{\cosh(p)}\overline{\widehat{b}_{k+p}} \\
=&- \alpha_+^{\pm}\frac{2\pi p^2}{\cosh(p)^2}\sum_{k \notin \{0,p,-p\}}\frac{k^2-\kappa_k(0)\kappa_p(0)}{\kappa_p(0) - \kappa_k(0)}\widehat{b}_{k-p}\overline{\widehat{b}_{k+p}}\\ &+\alpha_-^{\pm}\frac{2\pi p^2}{\cosh(p)^2}\sum_{\tilde k \notin \{0,p,-p\}}\frac{\tilde k^2-\kappa_{\tilde k}(0)\kappa_p(0)}{\kappa_p(0) - \kappa_{\tilde k}(0)}|\widehat{b}_{-\tilde k+p}|^2\\
=&-2\pi \overline{S_p(b)} \alpha_+^{\pm}+2\pi J_p(b)\alpha_-^{\pm},
\end{align*}
where we set $\tilde k=-k$ in the last sum. Using \eqref{coeffs_orthogonality_conditions_U_prime_1}, \eqref{coeffs_orthogonality_conditions_U_prime_2} and \eqref{coeffs_orthogonality_conditions_U_prime_3} (with the definition of $K_p$ given in Proposition~\ref{quasimodes_order_epsilon_nazarov_gap_formula}) we get that \eqref{orthogonality_condition_lambda''} can be written as \eqref{equation_vp_lambda_second}. 
\end{proof}
Let
$\xi_{p\pm}^{\app} := U^{\app}_{p\pm}\vert_{z=0}, \;\;
\tau_{p\pm}^{\app} := (1+\lambda_{p\pm}^{\app})^{-1}$, where 
 $ U^{\app}_{p\pm}, \lambda_{p\pm}^{\app}$ are defined just before \eqref{systeme_U_prime_ordre_2}.
The next result shows that $(\tau_{p\pm}^{\app},\xi_{p\pm}^{\app})$ are approximate eigenpairs of the resolvent operator.

\begin{proposition}\label{estimate_quasimodes_order_epsilon_2}
Under the assumptions of Theorem~\ref{one_gap-order-eps2} and assuming $\epsilon$ is small enough, we have for all $\delta\in [0,d_{p,5}]$
\begin{equation*}
 \norm{(1+G_{\delta \epsilon^2}[\epsilon b])^{-1}\xi_{p\pm}^{\app} - \tau_{p\pm}^{\app}\xi_{p\pm}^{\app}}_{L^2(\T_{2\pi})} \leq C_{b,p}\epsilon^3\norm{\xi_{p\pm}^{\app}}_{L^2(\T_{2\pi})}.
\end{equation*}
\end{proposition}

\begin{proof}
Inserting $U^{\app}_{p\pm} =U_{p\pm}^0 + \epsilon U'_{p\pm} + \epsilon^2 U''_{p\pm}$ in \eqref{definition_a_S}, we find that, for $V \in H^1(S)$, 
\begin{align*}
a^{R,S}_{\delta\epsilon^2,\epsilon}(U^{\app}_{p\pm} ,V) =&
\int_S P(\Sigma)\nabla U^{\app}_{p\pm} \cdot \nabla \overline{V} + \delta^2\epsilon^4 \int_S (1-\epsilon b(x)) U^{\app}_{p\pm} \overline{V} \\
&+ \int_\gamma U^{\app}_{p\pm} \overline{V} + i\delta\epsilon^2 \int \left(e_1 +\epsilon\begin{pmatrix}
-b(x) \\
z b'(x)
\end{pmatrix}\right) \cdot( U^{\app}_{p\pm} \nabla \overline{V}-\overline{V}\nabla U^{\app}_{p\pm} ) \\
=& (\lambda_{p\pm}^{\app}+1)\int_\gamma U^{\app}_{p\pm}\overline{V} + \tilde{L}(V)
\end{align*}
where ($\epsilon^3\widetilde{Q}
=P(\Sigma)-I_2 -\epsilon Q_1 - \epsilon^2 Q_2$)
\begin{align*}
&\tilde{L}(V) :=
-\epsilon^3 \lambda_{p\pm}''(\delta)\int_\gamma (U'_{p\pm} + \epsilon U''_{p\pm})\overline{V}
+
\epsilon^3\int_S \widetilde{Q}\nabla U^{\app}_{p\pm}\cdot \nabla \overline{V} \\
&+ \epsilon^3 \int_S Q_2 \nabla(U'_{p\pm} + \epsilon U''_{p\pm})\cdot\nabla \overline{V}
+ \epsilon^3\int_S Q_1 \nabla U''_{p\pm}\cdot\nabla \overline{V} 
+ \delta^2\epsilon^4\int_S (1-\epsilon b)U^{\app}_{p\pm}\overline{V}\\
&
+ i\delta\epsilon^3\int_S \begin{pmatrix}
-b(x) \\
z b'(x)
\end{pmatrix} \cdot (U^{\app}_{p\pm} \nabla \overline{V} - \overline{V} \nabla U^{\app}_{p\pm}) + i\delta\epsilon^3\int_S (U'_{p\pm} + \epsilon U''_{p\pm})\partial_x \overline{V} - \overline{V}\partial_x (U'_{p\pm} + \epsilon U''_{p\pm}) .
\end{align*}
Proceeding as we did to prove \eqref{norme_quasimode_non_normalise_ordre_epsilon}, we get 
\begin{align*}
&\norm{(1+G_{\delta \epsilon^2}[\epsilon b])^{-1}\xi_{p\pm}^{\app} - \tau_{p\pm}^{\app}\xi_{p\pm}^{\app}}_{L^2(\T_{2\pi})} \leq C\norm{\tilde{L}}_{H^1(S)'} \\
& \qquad\leq C_{b,p}\epsilon^3(\norm{U_{p\pm}^0}_{H^1(S)} + \norm{U'_{p\pm}}_{H^1(S)} + \norm{U''_{p\pm}(\delta)}_{H^1(S)}).
\end{align*} 
Using the operator $R_{2p,0}^{-1}$ (defined by \eqref{operateur-R_pinv}) we can show that
\begin{align*}
\norm{U'_{p\pm}}_{H^1(S)} &\leq C_{b,p}\norm{U^0_{p\pm}}_{H^1(S)} \\
\norm{U''_{p\pm}}_{H^1(S)} &\leq C_{b,p}\norm{U^0_{p\pm}}_{H^1(S)}, 
\end{align*} 
and therefore
\begin{equation*}
\norm{(1+G_{\delta \epsilon^2} [\epsilon b])^{-1}\xi_{p\pm}^{\app} - \tau_{p\pm}^{\app}\xi_{p\pm}^{\app}}_{L^2(\T_{2\pi})} \leq C_{b,p}\epsilon^3\norm{U_{p\pm}^0}_{H^1(S)}.
\end{equation*} 
Similar to \eqref{borne_inf_norme_U_app_ordre_epsilon}, if $\epsilon$ is small enough, $\norm{U_{p\pm}^{\app}} \geq \frac{1}{2}\norm{U_{p\pm}^0}$,
and
\begin{equation*}
\norm{\xi_{p\pm}^{\app}}_{L^2(\T_{2\pi})} \geq c_{b,p}\norm{U^0_{p\pm}}_{H^1(S)}
\end{equation*}
Proposition~\ref{estimate_quasimodes_order_epsilon_2} follows.
\end{proof}

We are now ready to establish the existence of a gap as we did at the end of Section~\ref{section:4}. 
For $\delta\in [0,d_{p,5}]$, hence $\theta \in [0,d_{p,5}\epsilon^2]$, Lemma~\ref{lemma:quasimodes} gives that 
\begin{align*}
 & \lambda_{2p}^\epsilon(\theta) \geq \lambda_{2p}^0(0) + \epsilon^2 J_p(b) + |S_p(b)|\epsilon^2 - C_p\epsilon^{3}, \\
 & \lambda_{2p-1}^\epsilon(\theta) \leq \lambda_{2p}^0(0) + \epsilon^2 J_p(b) - |S_p(b)|\epsilon^2 + C_p\epsilon^{3}.
\end{align*}
This gap remains open for $\theta\in[d_{p,5}\epsilon^2,d_{p,1}\epsilon]$ by Proposition~\ref{single-ev-higher-order} and for $\theta\in[d_{p,1}\epsilon,\frac12-d_{p,1}\epsilon]$ by Proposition~\ref{simple-ev}.
The precise size of the gap is derived by considering small $\delta$ and $\theta$, say $\theta=\delta \epsilon^2\leq \epsilon^3$, which conclude the proof of Theorem~\ref{one_gap-order-eps2}. Note that the center of the gap is displaced from the unperturbed value $\lambda_{2p}^0(0)$.

\section*{Acknowledgements}
We thank Antoine Venaille for pointing out to us several physics references.

This material is based upon work supported by the Swedish Research Council under grant no. 2021-06594 while the authors were in residence at the Institut Mittag-Leffler in Djursholm, Sweden, during the program “Order and Randomness in Partial Differential Equations” in the fall 2023. CL is partially supported by the BOURGEONS project grant ANR-23-CE40-0014-01 of the French National Research Agency (ANR). MM acknowledges financial support from the European Union (ERC, PASTIS, Grant Agreement n$^\circ$101075879). Views and opinions expressed are however those of the authors only and do not necessarily reflect those of the European Union or the European Research Council Executive Agency. Neither the European Union nor the granting authority can be held responsible for them. CS is partially supported by the Natural Sciences and Engineering Research Council of Canada (NSERC) under grant no.2018-04536. The authors thank the IDEX Universit\'e Grenoble Alpes for funding MM mobility grant.

\appendix
\section{Proof of Lemma~\ref{calcul-des-3-integrales} }\label{app-A}

Write $\Phi_p(x,z) = \rho_p(z)e^{ipx}$ with $\rho_p(z) = \frac{\cosh(p(z+1))}{\cosh(p)}$. We have $\nabla \Phi_p = \begin{pmatrix}
 ip\rho_p(z) e^{ipx} \\ \rho_p'(z) e^{ipx}
\end{pmatrix}$.
From $Q_1$ given in \eqref{Q1-Qk}, we have
\begin{align} 
&Q_1 \nabla \Phi_p = \begin{pmatrix}
 -ip\rho_p(z)b(x) + z\rho_p'(z)b'(x) \\ ipz\rho_p(z)b'(x) + \rho_p'(z)b(x) 
\end{pmatrix}e^{ipx} \label{Q1nablaPhip} \\
&Q_1 \nabla \Phi_p \cdot \nabla \overline{\Phi_p} = b(x)(\rho_p'(z)^2 - p^2 \rho_p(z)^2). \nonumber
\end{align}

Since $\int_0^{2\pi} b(x) dx =0$, we have $\int Q_1 \nabla \Phi_p \cdot \nabla \overline{\Phi_{p}} = 0$.
Now let us compute
\begin{align*}
Q_1 \nabla \Phi_p \cdot \nabla \overline{\Phi_{-p}} &= Q_1 \nabla \Phi_p \cdot \nabla \Phi_p \\
&= b(x)(\rho_p'(z)^2 + p^2 \rho_p(z)^2)e^{2ipx} +2ipz\rho_p(z)\rho'_p(z)b'(x)e^{2ipx},
\end{align*}
and notice that $2ip\widehat{b'}_{-2p} = 4p^2 \widehat{b}_{-2p}$ and that
\begin{align*}
 \int_{-1}^0 z\rho_p'(z)\rho_p(z)\dd z &= \frac{1}{2}\bigg[z\rho_p^2(z)\bigg]_{-1}^0 -\frac{1}{2}\int_{-1}^0 \rho_p^2(z) \dd z.
\end{align*}
Therefore,
\begin{align*}
\int_S Q_1 \nabla \Phi_p \cdot \nabla \overline{\Phi_{-p}} &= 2\pi\overline{\widehat{b}_{2p}}\int_{-1}^0 \rho_p'(z)^2 - p^2\rho_p(z)^2 \dd z + 2\frac{p^2}{\cosh^2(p)}2\pi \overline{\widehat{b}_{2p}}.
\end{align*}
Finally, 
\begin{align*}
\rho_p'(z)^2 - p^2 \rho_p(z)^2 = \frac{p^2}{\cosh^2(p)}(\sinh^2(p(z+1)) - \cosh^2(p(z+1))) = -\frac{p^2}{\cosh^2(p)}.
\end{align*}
We have thus obtained \eqref{integral_Q1_nabla_U_0_nabla_phi_q} and its conjugate \eqref{integral_Q1_nabla_U_0_nabla_phi_q2}.

\section{Proof of Lemma~\ref{formula_U_prime_ordre_2}}\label{app-B}

From \eqref{Q1nablaPhip} and its complex conjugate, we have
\begin{align*}
&Q_1 \nabla U_{p\pm}^0 = \alpha^{\pm}_+ Q_1\nabla\Phi_p + \alpha^{\pm}_- Q_1\nabla\Phi_{-p} \\
&\; = \alpha^{\pm}_+\begin{pmatrix}
-ip\rho_p(z)b(x) + z\rho_p'(z)b'(x) \\
ipz\rho_p(z)b'(x) + \rho_p'(z)b(x)
\end{pmatrix}e^{ipx} 
+ \alpha^{\pm}_- \begin{pmatrix}
ip\rho_p(z)b(x) + z\rho_p'(z)b'(x) \\
-ipz\rho_p(z)b'(x) + \rho_p'(z)b(x)
\end{pmatrix}e^{-ipx},
\end{align*}
hence
\begin{align*}
\div(Q_1 \nabla U_{p\pm}^0)
=& \alpha^{\pm}_+\big(2p^2\rho_p(z)b(x) + 2ipz\rho_p'(z)b'(x) 
+ z\rho_p'(z)b''(x)\big)e^{ipx} \\
&+ \alpha^{\pm}_-\big(2p^2\rho_p(z)b(x) - 2ipz\rho_p'(z)b'(x)
+ z\rho_p'(z)b''(x)\big)e^{-ipx}.
\end{align*}
A direct calculation on $E_{\pm p} = -z\rho_p'(z)b(x)e^{\pm ipx}$ shows that 
\begin{equation} \label{laplacien-E-p}
-\Delta E_p(x,z) =\div(Q_1 \nabla \Phi_p), \; 
-\Delta E_{-p}(x,z) =\div(Q_1 \nabla \Phi_{-p}),
\end{equation}
hence $-\Delta E_\pm = \div (Q_1 \nabla U_{p\pm}^0)$, with $E_\pm$ defined in \eqref{fonction-auxil-G}.
It is thus natural to introduce the function $V=U'_{p\pm}-E_\pm$. We find $-\Delta V = 0 $.
In order to rewrite \eqref{systeme_U_prime_ordre_2} as a system for $V$, we compute the boundary conditions. On the one hand, 
\begin{equation*}
\partial_z E_\pm = -(z\rho_p''(z) + \rho_p'(z))b(x)(\alpha^{\pm}_+ e^{ipx} + \alpha^{\pm}_- e^{-ipx})
\end{equation*}
implies
\begin{equation}\label{B2}
\begin{aligned}
\partial_z E_\pm(x,-1) &= \frac{p^2}{\cosh(p)}b(x)(\alpha^{\pm}_+ e^{ipx} + \alpha^{\pm}_- e^{-ipx}) \cr
\partial_z E_\pm(x,0) &= -\kappa_p(0) b(x)(\alpha^{\pm}_+ e^{ipx} + \alpha^{\pm}_- e^{-ipx}).
\end{aligned}
\end{equation}
On the other hand, 
\begin{align*}
-Q_1 \nabla U_{p\pm}^0\cdot e_z =& -\alpha^{\pm}_+(ipz\rho_p(z)b'(x) + \rho_p'(z)b(x))e^{ipx} \\
& - \alpha^{\pm}_-(-ipz\rho_p(z)b'(x) + \rho_p'(z)b(x))e^{-ipx}
\end{align*}
gives
\begin{equation}\label{B3}
\begin{aligned}
&-Q_1 \nabla U_{p\pm}^0\cdot e_z (x,-1) = \alpha^{\pm}_+\frac{ip}{\cosh(p)}b'(x)e^{ipx} - \alpha^{\pm}_-\frac{ip}{\cosh(p)}b'(x)e^{-ipx} \cr
&-Q_1 \nabla U_{p\pm}^0\cdot e_z(x,0) = -\kappa_p(0) b(x)(\alpha^{\pm}_+ e^{ipx} + \alpha^{\pm}_- e^{-ipx}).
\end{aligned}
\end{equation}
We finally find that $V$ satisfies the system :
\begin{equation*}
\left\{
\begin{aligned}
 -\Delta V =& 0 \qquad \textrm{in} \; S \\
 \partial_z V =& \alpha_+^{\pm}\bigg(\frac{ip}{\cosh(p)}b'(x)-\frac{p^2}{\cosh(p)}b(x)\bigg)e^{ipx} \\
&+ \alpha_-^{\pm}\bigg(-\frac{ip}{\cosh(p)}b'(x)-\frac{p^2}{\cosh(p)}b(x)\bigg)e^{-ipx} \qquad \textrm{on} \; \{z = -1\} \\
 \partial_z V =& \kappa_p(0) V \qquad \textrm{on} \; \{z = 0\}.
\end{aligned}\right.
\end{equation*}

Denoting $V_k(z) = \widehat{V}(k,z)$, the Fourier coefficients of $V$ in the $x$-variable, we have
\begin{equation*}
\left\{
\begin{aligned}
V_k(z) =& \beta_k\cosh(k(z+1)) + \gamma_k\sinh(k(z+1)) \\
V_k'(0) =& \kappa_p(0) V_k(0), \;\;
V_k'(-1) = \frac{kp}{\cosh(p)}(-\alpha_+^{\pm}\widehat{b}_{k-p}+\alpha_-^{\pm}\widehat{b}_{k+p}).
\end{aligned}\right.
\end{equation*}

For $k \neq 0,\pm p $,
\begin{equation*}
\gamma_k = \frac{p}{\cosh(p)}(-\alpha_+^{\pm}\widehat{b}_{k-p}+\alpha_-^{\pm}\widehat{b}_{k+p}), \;\; 
\beta_k = \frac{k^2 - \kappa_k(0)\kappa_p(0)}{k(\kappa_p(0) - \kappa_k(0))}\gamma_k.
\end{equation*}
When $k=\pm p$, $\gamma_{\pm p} = 0$, but there is no condition on $\beta_{\pm p}$. It means that
the function $(x,z)\mapsto \cosh(\pm p(z+1))e^{\pm ipx}$ belongs to the kernel of the system, and one solution is given by choosing $\beta_{p}=\beta_{-p}=0$. Finally, when $k=0$, $V_0'(-1) = \kappa_p(0) V_0$ gives $\beta_0=0$.

\section{Proof of \eqref{claim} }\label{app-C}

For $\beta \in \{+,-\}$, we have by \eqref{laplacien-E-p}
\begin{align*}
\int_S& \Delta E_\pm E_{\beta p} = -\int_S \div(Q_1 \nabla U_{p\pm}^0) E_{\beta p} \\
=& \alpha_+^\pm \int_S \big(2p^2\rho_p(z)b(x) + 2ipz\rho_p'(z)b'(x) + z\rho_p'(z)b''(x)\big)z\rho_p'(z)b(x)e^{i(p+\beta p) x} \dd x \dd z \\
&+ \alpha^{\pm}_-\int_S \big(2p^2\rho_p(z)b(x) - 2ipz\rho_p'(z)b'(x) 
+ z\rho_p'(z)b''(x)\big)z\rho_p'(z)b(x) e^{i(\beta p-p) x} \dd x \dd z \\
=:& \alpha_+^\pm {\textrm{I}}_+^\beta + \alpha^{\pm}_- {\textrm{I}}_-^\beta
\end{align*}
Now we compute $Q_2\nabla U_{q\pm}^0$ where $Q_2=\begin{pmatrix}
 0 & 0 \\ 0 & b(x)^2 + (zb'(x))^2
\end{pmatrix}$.
We find
\begin{multline*}
\int_S Q_2\nabla U_{p\pm}^0 \cdot \nabla \Phi_p 
= \alpha^{\pm}_+ \int_S \rho_p'(z)^2(b^2(x) + (zb'(x))^2)e^{i(p+\beta p) x}\dd x \dd z \\
+ \alpha^{\pm}_- \int_S \rho_p'(z)^2(b^2(x) + (zb'(x))^2) e^{i(\beta p-p) x} \dd x \dd z = :
 \alpha_+^\pm {\textrm{II}}_+^\beta + \alpha^{\pm}_- {\textrm{II}}_-^\beta .
\end{multline*}

Using $\int_{0}^{2\pi} b(x)b'(x)\dd x=0$, $\int_{0}^{2\pi} b(x)b''(x)\dd x=-\int_{0}^{2\pi} (b'(x))^2\dd x$ and that
\begin{gather*}
 2ip\int_0^{2\pi}b(x) b'(x) e^{2ipx}\dd x=-\int_0^{2\pi} (b(x) b''(x) + (b'(x))^2) e^{2ipx}\dd x\\
 \int_{-1}^0 2p^2z\rho_p(z)\rho_p'(z) \dd z=\int_{-1}^0 2z\rho_p''(z)\rho_p'(z) \dd z= -\int_{-1}^0 (\rho_p'(z))^2 \dd z
\end{gather*}
we show that $ \textrm{I}_+^\beta +
\textrm{II}_+^\beta = \textrm{I}_-^\beta +
\textrm{II}_-^\beta =0$, implying \eqref{claim}.

\end{document}